\documentclass[final,a4paper,12pt]{article}
\usepackage[a4paper]{geometry}
\usepackage[utf8]{inputenc}
\usepackage[T1]{fontenc}
\usepackage{authblk}
\usepackage{amsmath,amsbsy,amscd,amssymb,graphicx,epsfig,color, amsfonts, amsthm}
\graphicspath{{./figures/}}
\usepackage{bbding}
\usepackage{enumerate}
\usepackage[]{subfigure}
\usepackage{float,caption}
\usepackage{etex}
\usepackage{latexsym,mathrsfs}
\usepackage{fancybox,shadow, color,psfrag}
\usepackage{tikz}
\usepackage{hyperref}
\usepackage{wasysym}
\usepackage{cancel}
\usepackage{multimedia}
\usepackage{xmpmulti}
\usepackage{pstricks,pst-node,pst-text,pst-3d}
\usepackage{moreverb,epsfig,color,subfigure}

\DeclareMathOperator*{\argmin}{arg\,min}

\newtheorem{theorem}{Theorem}[section]

\newtheorem{lemma}[theorem]{Lemma}
\newtheorem{thm}{Theorem}[section]

\newtheorem{rem}[thm]{Remark}

\title{Identification of the heterogeneous conductivity in an inverse heat conduction problem}

\author[1]{Angel Ciarbonetti\thanks{Corresponding author: aciarbonetti@santafe-conicet.gov.ar}}
\author[2]{Sergio Idelsohn\thanks{sergio@cimne.upc.edu}}
\author[1,3]{Ruben D. Spies\thanks{rspies@santafe-conicet.gov.ar}}
\affil[1]{Instituto de Matem\'{a}tica Aplicada del Litoral, IMAL, CONICET-UNL, Centro Cient\'{\i}fico Tecnol\'{o}gico CONICET-Santa Fe,
Argentina}
\affil[2]{CIMNE International Centre for Numerical Methods in Engineering, Barcelona, Spain}
\affil[3]{Departamento de Matem\'{a}tica, Facultad de Ingenier\'{\i}a Qu\'{\i}mica, Universidad Nacional del Litoral, Santa Fe, Argentina}

\begin{document}
\maketitle
\begin{abstract}
This work deals with the problem of determining a non-homogeneous heat conductivity profile in a steady-state heat conduction boundary-value problem with mixed
Dirichlet-Neumann boundary conditions over a bounded domain in $\mathbb{R}^n$, from the knowledge of the state over the whole domain. We develop a method based on a
variational approach leading to an optimality equation which is then projected into a finite dimensional space. Discretization yields a linear although severely
ill-posed equation which is then regularized via appropriate ad-hoc penalizers resulting a in a generalized Tikhonov-Phillips functional. No smoothness assumptions
are imposed on the conductivity. Numerical examples for the case in which the conductivity can take only two prescribed values (a two-materials case) show that the
approach is able to produce very good reconstructions of the exact solution.
\end{abstract}

%

\section{Introduction}
\label{intro}
The study, analysis and numerical solution of \textit{Inverse Heat Transfer Problems} (IHTP), has a significant role in a large number of engineering applications.
Aerospace, mechanical, chemical and nuclear engineering, materials design, are just few of the many areas in which developing appropriate and efficient tools for
solving IHTP is of special interest.

In IHTP in general, one seeks to estimate unknown functions or parameters, for example unknown flow conditions or physical parameters of the material, from indirect
measurements of the state, that is of the temperature field, which more often than not, in practical cases is contaminated by noise and/or can be acquired only at a
few discrete points over the domain under study.


One of the engineering branches that has aroused the greatest interest in IHTP is electronic engineering, where thermal dissipation problems inside equipments or
components are usually studied. These types of problems are generally associated with thermal conduction processes (see \cite{convex2004}, \cite{JANICKI199851}).
This, like other growing branches in the last two decades, have generated a genuine interest in deepening the study of the thermal problems associated with the
design of materials and thermal devices, seeking new material design methodologies, as well as efficient and cost-effective solutions.

Different methodological approaches to deal with the problem of designing thermal materials and devices in heterogeneous media having their roots in mathematical
modeling, numerical solutions and computational mechanics have recently emerged. Among them, perhaps the most widely known is the computational design of thermal
materials in heterogeneous media. This methodology arises from the application of optimization methods which were originally designed for solving solid mechanics
problems (\cite{Bendsoe2003}, \cite{bendose2003_1}). Although this technique has shown to yield reasonable results (see for instance \cite{fachinotti016},
\cite{PERALTA2017}, \cite{fachinotti2018}), it has some severe drawbacks both in regard to the uniqueness of solutions (which appears to be highly dependent on the
set of parameters used in the optimizing algorithm), as well as from the post-processing requirements of the obtained results.

Taking as inspirational base the optimization methods approach, we propose to tackle the problem of design of thermal materials in heterogeneous media by embedding
it within the IHTP framework.

Undoubtedly the largest difficulty in dealing with IHTP problems is the fact that they are severely ill-conditioned which is a direct consequence of the violation of
Hadamard's third postulate (\cite{ref:Hadamard-1902}): lack of continuous dependence of the true solution on the data. Quite often there are additional complications
coming from the lack of good measurements of the temperature field or the the fact they are only available at a few scarce points. That makes extremely difficult to
obtain good estimates of the heat flow which is generally needed for solving most IHTP.

Several variational models widely used for the numerical solution of IHTP generally introduce simplification hypotheses on the constitutive models and/or on the
boundary conditions. For example in \cite{chinyu98}, the authors consider conductivity to be a polynomial form of the temperature. Similarly in
\cite{MIERZWICZAK2011}, the model used for the IHTP problem assumes a linear dependence between conductivity and temperature. Although these simplifying hypotheses
allow the successful implementation of conventional schemes for obtaining good numerical solutions, they are far from reflecting most real world problems.


In our case, the IHTP is associated with identifying an unrestricted heterogeneous conductivity profile in two dimensions from measurements of the temperature in a
steady-state regime. From the materials design approach, an arbitrary field of temperatures and boundary conditions can be initially imposed so that they be
consistent with the constraints of the expected design solution. 

Heterogeneity adds a considerable degree of difficulty to an IHTP, however, that assumption corresponds to the problems of greatest practical interest, both in the
analysis of thermographic images as in the material design problem in which we focus this contribution.

Implementation and numerical resolution of any IHTP will, at some point, require solving an optimization problem combined with some type of regularization to deal
with ill-posedness.

In \cite{huang2000two} the authors deal with the problem of identifying the conductivity profile from the temperature map. A simplifying assumption consisting of a
direct dependence on conductivity and temperature is made together with pure Neumann-type boundary conditions. Under these hypotheses the authors solve the problem
using a conjugate gradient approach. A similar approach using conjugate gradient is used in \cite{rodriguez2012inverse} where the authors solve the inverse problem
of finding the geometry of a defect or inclusion from an IR image.

In this work, a novel approach for solving the identification of an heterogeneous conductivity profile based on a variational approach of the PDE model followed by
appropriate discretization and regularization is introduced. The model allows for mixed Dirichlet-Neumann boundary conditions, and no restrictive assumptions are
made on the conductivity profile. The approach is numerically tested with 2D examples although the setting is not restricted to two dimensions.

\subsection{A brief historical mathematical tracking of the problem}
In this article we consider the problem of determining the elliptic coefficient profile function in an homogeneous elliptic boundary value problem.

Several authors have worked on this type of problems before, as they appear in several areas and concrete applied problems such as electrical conductivity problems,
oil resevoir and ground water flow problems (\cite{refb:Bear-1972}, \cite{ref:Bongiorno-Valente-1977}, \cite{Knowles-Wallace-1996}, \cite{Kohn-Lowe1988},
\cite{Yeh-1986}) among others.

In a 1980 article (\cite{Calderon-1980}), A. P. Calder\'{o}n considered the following problem. Let $\Omega \subset \mathbb{R}^n$, $n\ge 2$ be a bounded domain with
Lipschitz boundary $\partial \Omega$ and define
\begin{equation}\label{domain-of-M}
    L_{>0}^\infty(\Omega)\;\doteq\;\left\{ k \in L^\infty(\Omega) \;:\; \exists\,\epsilon>0 \textrm{ such that }k(x)\ge\epsilon>0,\; \forall x\in \Omega\right\}.
\end{equation}
Further, for $k \in L_{>0}^\infty(\Omega)$ let $L_k$ be the differential operator
\begin{equation}\label{operator}
    L_k(u)\;=\;\nabla.(k \nabla u),
\end{equation}
with domain $D(L_k)\doteq H^1(\Omega)$ and the quadratic form $Q_k:D(Q_k)\subset L_\infty(\partial \Omega) \to \mathbb{R}_0^+$, defined by the Dirichlet integral
\begin{equation} \label{Qgamma}
    Q_k(\varphi)=\int_\Omega k(x)\left(\nabla u_\varphi(x)\right)^2\;dx\;=\; \int_{\partial\Omega}\varphi(x)k(x)\frac{\partial u_\varphi}{\partial \nu}\; ds,
\end{equation}
where $\nu$ denotes the outward normal to $\partial \Omega$ and $u_\varphi\in H^1(\Omega)$ is the solution of the Dirichlet boundary value problem
\begin{equation} \label{IBVP}
    \begin{cases} L_k(u)=0, &x\in\Omega, \\
        u=\varphi, & x\in\partial\Omega.
    \end{cases}
\end{equation}
It is timely to mention here that for $\varphi$ in the space of traces of functions of $H^1(\Omega)$, the IBVP (\ref{IBVP}) has in fact a unique solution $u \in
H^1(\Omega)$ (see \cite{Stampacchia-1966}).

The problem considered by Calder\'{o}n was to decide wether the function $k$ is uniquely determined by the quadratic form $Q_k$ and, if that is true, try to compute $k$
in terms of $Q_k$. He was able to show that, with an appropriate norm defined in the space of quadratic forms, the mapping
\begin{equation}\label{mappingM}
    M:k\longrightarrow Q_k
\end{equation}
is bounded and analytic in $L_{>0}^\infty(\Omega)$. Moreover, although the general problem escaped its proof at that time, Calder\'{o}n was able to show that for the
linearized problem, the answer is affirmative, that is, $d\,M |_{k=const.}$ is an injective mapping. In addition he showed that if the function $k$ is ``close
enough'' to being a constant, then $k$ is ``nearly'' determined by its quadratic form $Q_k$ and derived a bound for the $L^\infty$-norm of the error (see
\cite{Calderon-1980} for details).

The previous result for ``nearly constant'' functions was later extended and formalized by Sylvester and Uhlmann (\cite{Silvester-Uhlmann-1986}) only for the case
$n=2$.

In 1984, Kohn and Vogelius (\cite{Kohn-Vogelius-1984}) showed that if $k$ is real analytic, then it can be uniquely determined from the knowledge of its Dirichlet
integral $Q_k$. This result was later extended to piecewise real analytic functions $k$ in 1985 by the same authors (\cite{Kohn-Vogelius-1985}).

In 1987, Sylvester and Uhlmann (\cite{Silvester-Uhlmann-1987}) showed that for functions $k$ sufficiently smooth, the quadratic form $Q_k$ does indeed uniquely
determine the function $k$. More precisely, they showed that the mapping $M$ in (\ref{mappingM}) is injective over $C^\infty(\overline\Omega)\cap
L_{>0}^\infty(\Omega)$.

The case of determining $k$ in (\ref{IBVP}) from information about $u$ on the whole domain $\Omega$ was studied by several authors. It is clear that some assumptions
on $k$ and/or on $u$ must be required. Thus, for instance, it is evident that $k$ cannot be uniquely determined in any subregion of $\Omega$ where $\nabla u=0$.
However, if $|\nabla u| >0$ everywhere on $\Omega$, once the values of $k$ are given on a hypersurface transversal to $\nabla u$, the method of characteristics will
yield a unique solution. Several authors obtained similar uniqueness results under weaker assumptions on $\nabla u$ under diverse assumptions on $k$ (e.g.
\cite{Alessandrini-1984}, \cite{Alessandrini-1986}, \cite{Bongiorno-Valente-1977}, \cite{Richter-1981}).

There are several articles devoted to the problem of recovering $k$ from information about $u$. However all of them assume some degree of smoothness on $k$ (at least
differentiability), which is almost never true in practical applications, where at best, only piecewise smoothness and jump discontinuities are to be expected, and
quite often the available data consists only of noisy measurements at some discrete points. Although the mathematical theory of elliptic equations with discontinuous
principal coefficients is well known (\cite{Stampacchia-1966}, \cite{Littman-Stampacchia-Weinberger-1963}), there is not much done on the inverse problem of
recovering $k$ in these cases.

In this article we develop a method for approximating the perhaps discontinuous principal coefficient function based on a regularized variational approach. This
approach takes appropriate care of the strong ill-posedness of the inverse problem which has been already carefully reported by several authors (see for instance
\cite{Yakowitz-Duckstein-1980}, \cite{refb:Engl-Hanke-96}), and is characteristic in all inverse heat conduction problems. Several numerical examples are presented
that show that the method is able to yield very good approximations of $k$ even in the case of discontinuous profiles.

\section{Methodology }
\subsection{Preliminaries}

Let $\Omega \subset \mathbb{R}^n$ ($n\ge 2$) a bounded open set with smooth boundary $\Gamma=\partial\Omega \,=\,\overline\Gamma_1\cup\overline\Gamma_N$,
with $\Gamma_1\cap\Gamma_N=\emptyset$, $ c, k, f, g, h \in L^2(\Omega)$ with $0<\gamma_1\le k(x)\le \gamma_2$ and $ c\ge 0$, and consider the following problem,
as described in Figure 1:
\begin{equation}\label{mainequation}
    {\cal{P}}={\cal{P}}(k,c,f,g,h):\;\begin{cases} -{\textrm{div}}(k(x)\nabla u(x))+c(x)u(x)=f(x), &x\in\Omega, \qquad (a) \\
        u(x)=g(x),&x\in\Gamma_D, \hfill(b) \\
        k(x)\,\nabla u(x) \cdot \vec n =h(x), &x\in\Gamma_N. \hfill (c)
    \end{cases}
\end{equation}
\begin{figure}[h!]\
    \centering
    \includegraphics[width=0.65\textwidth]{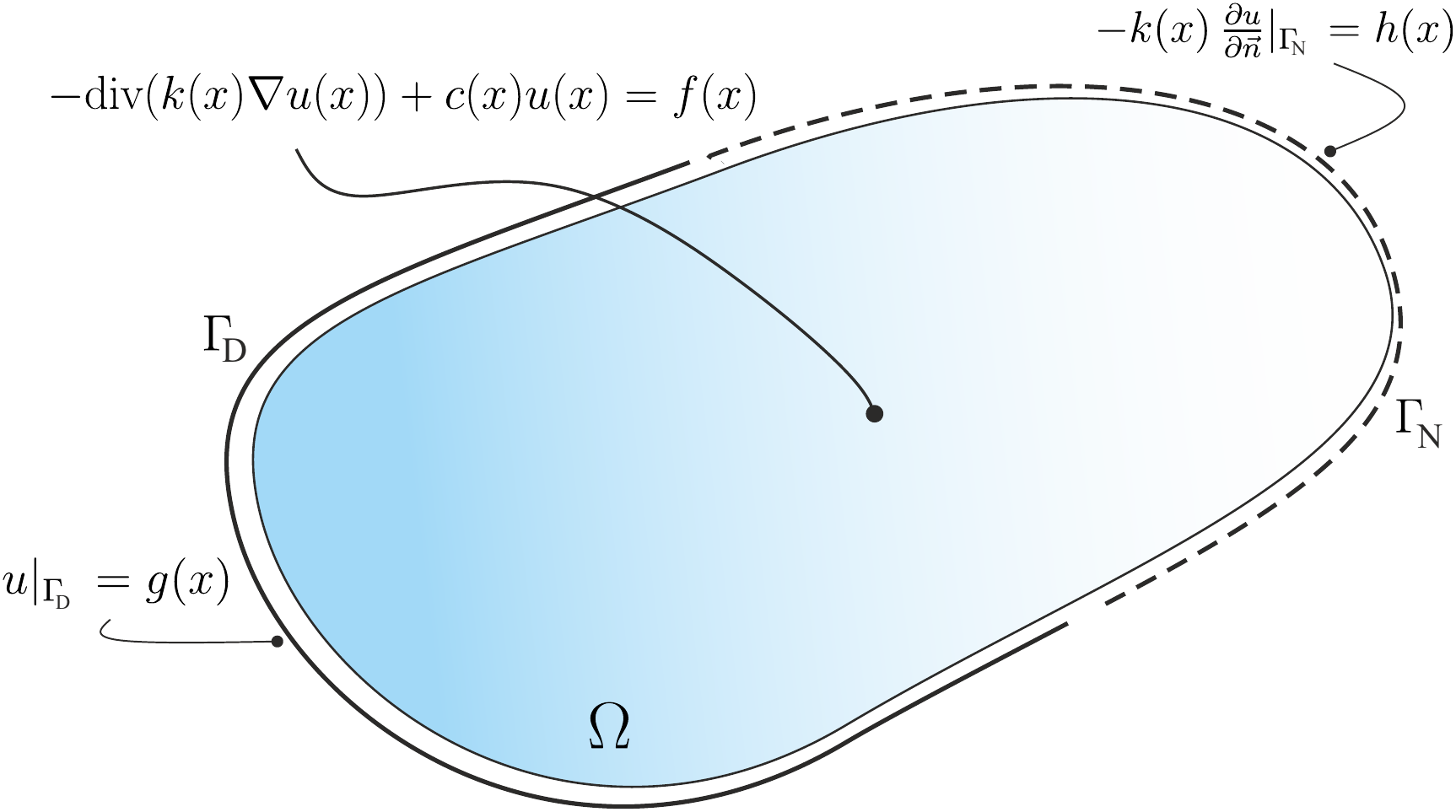} \\
    \caption{Schematic representation of problem ${\cal{P}}(k,c,f,g,h)$ }
    \label{figure1}
\end{figure}

Assume further that $g\in C(\Gamma_D)$ and define the affine subset of $ H^1(\Omega)$
\begin{equation}\label{affinesubset}
    H^1_{\Gamma_D,g}(\Omega)\;=\; \left\{ v\in  H^1(\Omega)\; :\; v|_{\Gamma_D}=g \right\}
\end{equation}

Multiplying equation $(a)$ in (\ref{mainequation}) by $v\in H^1_{\Gamma_D,0}(\Omega)$ and integrating we obtain:

\begin{align*}
    0&= -\int_\Omega\textrm{div}(k\nabla u)v\;dx +\int_\Omega cuv\;dx -\int_\Omega fv\;dx \\
    &=\int_\Omega\langle k\nabla u,\nabla v\rangle\;dx -\int_{\partial\Omega}\langle k\nabla u, \vec n\rangle v\; ds +\int_\Omega cuv\;dx -\int_\Omega fv\;dx \\
    &=\int_{\Omega}\left(\langle k \nabla u,\nabla v\rangle + cuv\right)\; dx - \int_\Omega fv\;dx - \int_{\Gamma_D\cup\Gamma_N}\langle k\nabla u, \vec n\rangle v\; ds \\
    &= \int_\Omega \left(\langle k \nabla u,\nabla v\rangle + cuv\right)\; dx - \int_\Omega fv\;dx \\ & \qquad - \int_{\Gamma_N} h v\; ds
    \qquad\qquad {\textrm{(since $ v|_{\Gamma_D=0}$ and  $ k\nabla u\cdot \vec n|_{\Gamma_N}=h$})} \\
    &\doteq F(u,v),
\end{align*}
where $\langle\cdot,\cdot\rangle$ denotes the usual inner product in $ L^2(\Omega)$ and $\vec n$ is the outer normal to $\partial \Omega$. Hence, the variational
formulation of problem ${\cal{P}}$ is the following.

\medskip
$VF(\cal{P})$: {\it Find $u$ in $H^1_{\Gamma_D,g}(\Omega)$ such that $F(u,v)=0$ for all $v\in H^1_{\Gamma_D,0}(\Omega)$, i.e. such that }
\begin{equation}\label{problem-wf}
    \int_\Omega \left(\langle k \nabla u,\nabla v\rangle + cuv\right)\; dx \;=\; \int_\Omega fv\;dx + \int_{\Gamma_N} h v\; ds, \qquad\forall v\in
    H^1_{\Gamma_D,0}(\Omega).
\end{equation}
Now define the continuous symmetric bilinear form $B_{k,c}:H^1(\Omega)\times H^1(\Omega) \to\mathbb{R}$ by
\begin{equation}\label{bilinearform}
    B_{k,c}(u,v)\;\doteq\; \int_\Omega\left( \langle k \nabla u,\nabla v\rangle + cuv \right)\; dx.
\end{equation}
Then, if $c>0$, $B_{k,c}$ defines an inner product on $H^1(\Omega)$ (note also that if $c>0$ then $B_{k,c}$ is $H^1(\Omega)$-elliptic) with associated norm
\begin{equation}\label{a-norm}
    \| u \|^2_{B_{k,c}} \;\doteq\;  \int_\Omega\left(k\|\nabla u\|^2 + c|u|^2\right)\; dx.
\end{equation}
Define also the energy functional $J:H^1_{\Gamma_D,g}(\Omega)\to \mathbb{R}$ by
\begin{align}\label{energyfunctional}
    J(v)\; &\doteq \; \frac12 B_{k,c}(v,v)-\int_\Omega fv\;dx -\int_{\Gamma_N} hv \;ds \nonumber \\
    &= \frac12\int_\Omega \left( \langle k \nabla v,\nabla v\rangle + c\,v^2 \right)\; dx-\int_\Omega fv\;dx -\int_{\Gamma_N} hv \;ds.
\end{align}
In the next result we show that the directional derivatives of this energy functional are characterized by the functional $F$ defined above.
\begin{lemma}\label{dtJ-equal-WF} For any $u\in H^1_{\Gamma_D,g}(\Omega)$ and any $v\in H^1_{\Gamma_D,0}(\Omega)$ there holds
    $$ \frac{d}{dt}J(u+tv){\Large|}_{t=0}\;=\; F(u,v).
    $$
\end{lemma}
\begin{proof}
    First note that if $u\in H^1_{\Gamma_D,g}(\Omega)$ and $v\in H^1_{\Gamma_D,0}(\Omega)$ then $u+tv\in H^1_{\Gamma_D,g}(\Omega)$ for all $t\in \mathbb{R}$. Moreover,
    $J(u+vt)$ is clearly differentiable with respect to $t$ and
    \begin{align}
        \frac{d}{dt}J(u+tv)\;&=\; \frac{d}{dt} \left[\frac12 B_{k,c}(u+tv,u+tv)-\int_\Omega (u+tv)\,f\;dx -\int_{\Gamma_N}(u+tv)\,h\;ds \right] \nonumber \\
        &=\frac{d}{dt}\left[\frac12\int_\Omega\left( \langle k\nabla(u+tv),\nabla(u+tv)\rangle +c(u+tv)^2 \right)\; dx -\int_\Omega fu\;dx\right. \nonumber \\
        &\qquad \quad \left.-t\int_\Omega fv\;dx -\int_{\Gamma_N}hu\;ds -t \int_{\Gamma_N}hv\;ds \right] \nonumber \\
        &=\int_\Omega\left( \langle k\nabla u,\nabla v\rangle + kt|\nabla v|^2+ cuv +tv^2 \right)\; dx \nonumber \\ & \qquad -\int_\Omega fv\;dx - \int_{\Gamma_N}hv\;ds.
    \end{align}
    Hence
    \begin{align}
        \frac{d}{dt}J(u+tv){\Large|}_{t=0}\;&=\; \int_\Omega\left( \langle k\nabla u,\nabla v\rangle + cuv \right)\; dx -\int_\Omega fv\;dx - \int_{\Gamma_N}hv\;ds
        \nonumber\\
        &=F(u,v).
    \end{align}
    \hfill
\end{proof}

The next theorem constitutes a fundamental result for the approximation framework that we shall develop later. It relates the solution of the variational problem
$VF(\cal{P})$ with the minimizer of the energy functional $J$.

%
\begin{theorem}\label{thm-existence} The variational formulation problem $VF(\cal{P})$ in (\ref{problem-wf}) does have a unique solution
    $u^*$ in $H^1_{\Gamma_D,g}(\Omega)$. Moreover, such a solution is characterized by the unique minimizer of the energy functional $J$ defined in
    (\ref{energyfunctional}), i.e.
    \begin{equation}
        u^*\; = \; \argmin_{u\in H^1_{\Gamma_D,g}(\Omega)}\; J(u).
    \end{equation}
\end{theorem}
\begin{proof}
    It is easy to show that the energy functional $J$ in (\ref{energyfunctional}) is strictly convex and coercive over $H^1_{\Gamma_D,g}(\Omega)$
    (for this we need $c$ to be strictly positive), which implies the existence and uniqueness of a global minimizer $u^* \in H^1_{\Gamma_D,g}(\Omega)$.
    Hence, for any $v \in H^1_{\Gamma_D,0}(\Omega)$ there must hold
    \begin{equation*}
        0\;=\; \frac{d}{dt}\, J(u^*+tv){\Large|}_{t=0}\;=\; F(u^*,v),
    \end{equation*}
where the last equality follows from Lemma \ref{dtJ-equal-WF}. Thus $F(u^*,v)=0$ for all $v \in H^1_{\Gamma_D,0}(\Omega)$ which implies that
    \begin{equation*}
        \int_\Omega\left( \langle k\nabla u^*,\nabla v\rangle + cu^*v \right)\; dx =\int_\Omega fv\;dx + \int_{\Gamma_N}hv\;ds  \qquad \forall v \in
        H^1_{\Gamma_D,0}(\Omega),
    \end{equation*}
    and therefore $u^*$ is a solution of problem $VF(\cal{P})$ given in (\ref{problem-wf}).

    Now let us prove that the solution of problem $VF(\cal{P})$ is
    unique. Let $u^*$ be a solution of problem $VF(\cal{P})$ and $w \in
    H^1_{\Gamma_D,g}(\Omega)$ be arbitrary. Then
    \begingroup
    \allowdisplaybreaks
    \begin{align*}
        J(w)&-J(u^*)\;=\;\frac12 B_{k,c}(w,w)-\int_\Omega fw\; dx -\int_{\Gamma_N} hw\; ds - \frac12 B_{k,c}(u^*,u^*) \\ &\qquad +\int_\Omega fu^*\; dx +\int_{\Gamma_N} hu^*\; ds\\
        &=\frac12\left(B_{k,c}(w,w)- B_{k,c}(u^*,u^*) \right) -\left(\int_\Omega f \underset{\in H^1_{\Gamma_D,0}(\Omega)}{\underbrace{(w-u^*)}}\;dx + \int_{\Gamma_N}
        h\underset{\in H^1_{\Gamma_D,0}(\Omega)}{\underbrace{(w-u^*)}}\;ds \right) \\
        &= \frac12\left( B_{k,c}(w,w)- B_{k,c}(u^*,u^*) \right) \\ &\quad -\int_\Omega\left(\langle k\nabla u^*,\nabla(w-u^*)\rangle + cu^*(w-u^*)\right)\;dx
        \quad{\textrm {(since $u^*$ solves $VF(\cal{P})$)}} \\
        &= \frac12\left( B_{k,c}(w,w)- B_{k,c}(u^*,u^*) \right) -B_{k,c}(u^*,w-u^*)\\
        &= \frac12 B_{k,c}(w,w)-\frac12 B_{k,c}(u^*,u^*) -B_{k,c}(u^*,w)+  B_{k,c}(u^*,u^*) \\
        &=\frac12\left( B_{k,c}(w,w) +B_{k,c}(u^*,u^*) -2B_{k,c}(u^*,w)\right)\\
        &=\frac12 B_{k,c}(w-u^*,w-u^*)\\
        &=\frac12 \| w-u^*\|_{B_{k,c}}^2.
    \end{align*}
    \endgroup
    It then follows that $J(w)\ge J(u^*)$ for all  $w \in H^1_{\Gamma_D,g}(\Omega)$ and $J(w)=J(u^*)$ if and only if $w=u^*$.

    Hence problem $VF(\cal{P})$ has a unique solution  $u^* \in H^1_{\Gamma_D,g}(\Omega)$ which is the unique minimizer of the energy functional $J(\cdot)$ over $
    H^1_{\Gamma_D,g}(\Omega)$. \hfill\end{proof}
\begin{rem}
    Note that within this variational formulation, the condition $c>0$ (strictly) is necessary for uniqueness.
\end{rem}
\begin{rem}
    The condition $k(x)\ge \gamma_1 >0$ is strictly necessary for  $J$ to be accretive. The condition $k(x)\le \gamma_2<\infty$ can be replaced by $k(x)$ bounded above
    by a positive function in $L^1(\Omega)$.
\end{rem}

\subsection{The inverse problem}

Now, given all the model parameters $\Omega$, $\Gamma_D$, $\Gamma_N$, $c$, $f$, $g$ and $h$, and a prescribed temperature distribution $\hat u\in
H^1_{\Gamma_D,g}(\Omega) $ we consider the problem of finding the corresponding distributed conductivity field $k(\cdot)$ such that $u^*=\hat u$. That is, we intend
to ``invert'' problem ${\cal{P}}={\cal{P}}(k,c,f,g,h)$ with respect to $k$.

For simplicity and clarity in the presentation of the numerical results only the 2D-case will be considered, i.e.\;we will take $n=2$, and we will assume $f=h\equiv
0$. Thus given a prescribed temperature field $\hat u(x,y)\in  H^1_{\Gamma_D,g}(\Omega)$ we want to find $k=k(x,y)$ such that $\hat u$ be the unique solution of
problem ${\cal{P}}(k,c,f,g,h)$. Then, according to Theorem \ref{thm-existence}, $k(x,y)$ must satisfy:

\begin{align}\label{eqnstar1}
    0\;&=\;\frac{d}{dt} J(\hat u +tv){\Large|}_{t=0}\nonumber \\
    &=\int_\Omega\left(\langle k\nabla \hat u,\nabla v \rangle +c \hat u v  \right)\;dx\,dy, \qquad \forall v\in H^1_{\Gamma_D,0}(\Omega).
\end{align}

\section{Implementation}

\subsection{Discretizing the optimality condition and regularization}
Let $\Omega_i$, $1\le i\le L$, be a partition of $\Omega$ by open sets, i.e. such that $\Omega_i\subset \Omega$ is open for all $i$,
$\displaystyle \cup_{i=1}^L\overline{\Omega_{i}}=\overline\Omega$ and $\Omega_i\cap\Omega_j=\emptyset$ if $i\ne j$. For each $i$ let
$(x_i,y_i) \in \Omega_i$ and for any function $q(x,y)$ defined on $\Omega$ let us denote with $q_i$ the value of $q$ at the point  $(x_i,y_i)$, i.e.
$q_i\doteq q(x_i,y_i)$. With this notation, equation (\ref{eqnstar1}) can then be approximated by
\begin{equation}\label{discrete1}
    0 = \sum_{i=1}^L \left[k_i\left(\hat u_{x,i}v_{x,i} + \hat u_{y,i}v_{y,i}\right) + c_i\hat u_i v_i\right] \,m(\Omega_i), \qquad \forall v\in
    H^1_{\Gamma_D,0}(\Omega),
\end{equation}
where the subscripts $x$ and $y$ denote the respective partial derivatives and $m(\Omega_i)$ is the Lebesgue measure of $\Omega_i$. Assuming that the partition is
regular so that $m(\Omega_i)$ is constant, equation (\ref{discrete1}) is equivalent to
\begin{equation}\label{discrete2}
    \sum_{i=1}^L k_i\left(\hat u_{x,i}v_{x,i} + \hat u_{y,i}v_{y,i}\right) = -\sum_{i=1}^L c_i\hat u_i v_i,\qquad \forall v\in H^1_{\Gamma_D,0}(\Omega).
\end{equation}
Consider now a finite, arbitrary set of functions
\begin{equation} \label{testfunctions}
    v^r\in H^1_{\Gamma_D,0}(\Omega),\qquad 1\le r\le R.
\end{equation}
Then (\ref{discrete2}) implies that
\begin{equation}\label{discrete3}
    \sum_{i=1}^L k_i\left(\hat u_{x,i}v^r_{x,i} +\hat
    u_{y,i}v^r_{y,i}\right) = -\sum_{i=1}^L c_i\hat u_i v^r_i,\qquad
    \forall\, 1\le r\le R.
\end{equation}
By defining
\begin{equation}\nonumber
    a_{r\ell}\;\doteq\; \hat u_{x,\ell}\,v^r_{x,\ell} + \hat u_{y,\ell}\,v^r_{y,\ell} \; \textrm{ and } f_r\;\doteq\; -\sum_{i=1}^L c_i\,\hat u_i\, v^r_i,\; 1\le \ell\le L,\;\; 1\le r \le R,
\end{equation}
we can write (\ref{discrete3}) as follows:
\begin{equation}\label{discrete4}
    \sum_{\ell=1}^La_{r\ell} k_\ell\;= f_r,\qquad \forall\, 1\le r\le R.
\end{equation}

Let now $\mathbf{A}$ be an $R\times L$-dimensional matrix, $K\in\mathbb{R}^L$ and $F\in\mathbb{R}^{R}$ whose elements are $a_{r\ell}$, $k_\ell$ and $f_r$, respectively.
Then, equation (\ref{discrete4}) can be simply written in matrix form as:
\begin{equation}\label{matrixform}
    \mathbf{A}\,K=F.
\end{equation}
Note that we still need to impose the condition that all components of the vector $K$ be bounded between the values $\gamma_1$ and $\gamma_2$. In what follows the
idea is to solve (\ref{matrixform}) in the least squares sense, weakly imposing this restriction through a penalizer.

With the above in mind we consider a generalized Tikhonov-Phillips functional of the form:
\begin{equation}\label{TP-functional}
    J_{\alpha,W}(K)\;\doteq\; \|\mathbf{A}K-F\|^2 + \alpha \,W(K),
\end{equation}
where $\alpha>0$ is a regularization parameter and the penalizer $W(K)$ must be appropriately constructed in such a way as to deter non-admissible values as well as
any undesired property of the conductivity profile $k(x,y)$.

\subsection{Assumptions for the numerical implementations}
In what follows we take $\Omega=(0,1)\times(0,1)$ and the Dirichlet and Newmann boundaries are the vertical and horizontal boundaries of $\Omega$, respectively, i.e.
$\Gamma_D=\{0,1\}\times (0,1)$ and $\Gamma_N= (0,1) \times \{0,1\}$ (see Figure \ref{discretization}).

Also, for the numerical experiments that follow, the Dirichlet boundary conditions will be prescribed as $g(x)=T_1$ on $\{0\}\times (0,1)$ and $g(x)=T_2$ on
$\{1\}\times (0,1)$ where $T_1$ and $T_2$ are prescribed constant temperature values, with $T_1>T_2$. Furthermore, we shall assume that $k(x,y)$ can take only two
possible values, say $k_L$ and $k_U$, with $0<k_L<k_U<\infty$. Physically this implies that only two different materials having those conductivities can be present
in the domain $\Omega$. Clearly this assumption is extremely important to design appropriate forms for $W(K)$. It is timely to mention however that all the previous
formulation leading up to equations (\ref{matrixform}) and (\ref{TP-functional}) only requires the assumption $k\in L_{>0}^\infty(\Omega)$.

\begin{figure}[h!]
    \centering
    \includegraphics[width=0.70\textwidth]{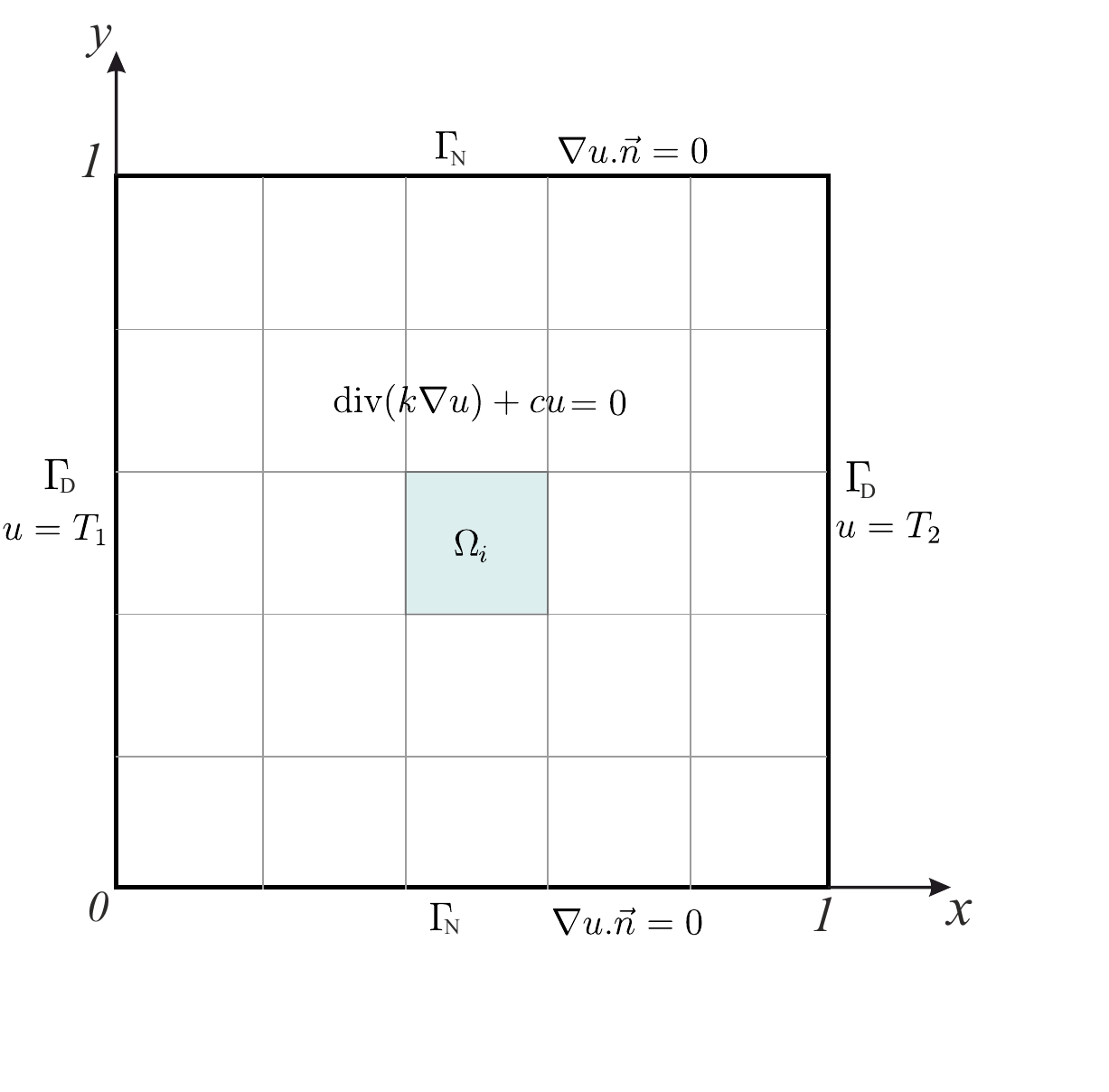}
    \caption{Schematic representation of the problem setting and its discretization}
    \label{discretization}
\end{figure}

\subsection{On the penalizer $\mathbf{W(K)}$}

Under the above assumption that at each point $(x,y)\in\Omega$,  $k(x,y)$ can only take one of the values $k_L$ or $k_U$, the penalizer $W(K)$ must be designed so
that it deters each and every component of the vector $K$ to take any but one of those two values. A simple way to do that is as follows.

Let $p:\mathbb{R}\to \mathbb{R}$ be the monic quadratic polynomial that vanishes at $k_L$ and at $k_U$, i.e. $p(z)=(z-k_L)(z-k_U)\;=\;z^2-(k_L+k_U)z+k_L k_U$. Next
we define $W_1:\mathbb{R}^{L}\to\mathbb{R}^+_0$ as \begin{equation}\label{penalizerW1} W_1(K)\;\doteq\; \|p(K)\|^2_{\mathbb{R}^L},
\end{equation}
where $p$ acting on the vector $K$ must be understood as its component-wise action, i.e. $\left( p(K)\right)_i =p(K_i)$. Clearly, then, $W_1(K)\ge 0$ and $W_1(K)=0$
if and only if each and every component of $K$ takes one of the two values $k_L$ or $k_U$.

A different approach is to design an ad-hoc penalizer which, besides promoting all the components of the vector $K$ to take only the values $k_L$ or $k_U$, provides
data-driven information about where to take one or the other value. With that in mind let $b_U \in \mathbb{R}^L$ be a binary vector whose $i^{th}$ component takes
the value $1$ if and only if the gradient of $\hat u$ at the point $(x_i,y_i)$ is ``large'', i.e. if $ \|\nabla \hat u(x_i,y_i)\|>\gamma$, where $\gamma$ is a
certain prescribed threshold value, and define the penalizer $W_2(K)$ as
\begin{equation}\label{penalizerW2}
    W_2(K)\doteq \|b_U\odot(K-k_L\,\mathbf{1})\|^2,
\end{equation}
where $\odot$ denotes elementwise product and $\mathbf{1}\in\mathbb{R}^L$ is a vector of ones. Hence, at a point where $ \|\nabla \hat u\|$ is large (as dictated by
the threshold parameter $\gamma$), this penalizer will discourage the corresponding component of $K$ to assume any value except $k_L$.

In the definitions (\ref{penalizerW1}) and (\ref{penalizerW2}), the $\mathbb{R}^L$-norm can be replaced by any other norm. Finally, many other options for the
functional $W(K)$ exist if other restrictions on the function $k(x,y)$ are to be enforced, for instance, certain regularity properties (which clearly would not be
appropriate for our case at hand).

\subsection{Examples, numerical experiments and discussions}

\subsubsection{Case I}

We first solved the forward problem $\cal{P}$ under the previous assumptions with $T_1 = 322.0\; [K]$, $T_2= 283.0\; [K]$, $c(x,y)=1.0=$\,constant, and $k(x,y)$ as
shown in Figure \ref{kxy-design1}. For this we used a standard discretization by finite element method, using biquadratic interpolation elements S2 with 8-nodes for
computing $\hat u(x,y)$, $\hat u_x(x,y)$ and $\hat u_y(x,y)$. The same scheme was used to solve all forward problems in the coming cases. The resulting temperature
distribution $\hat u (x,y)$, whose discretized values are then used as inputs for the inverse problem, is shown in Figure \ref{uhat1}\,b).

%
\begin{figure}[h!]
    \centering
    \includegraphics[width=1
    \textwidth]{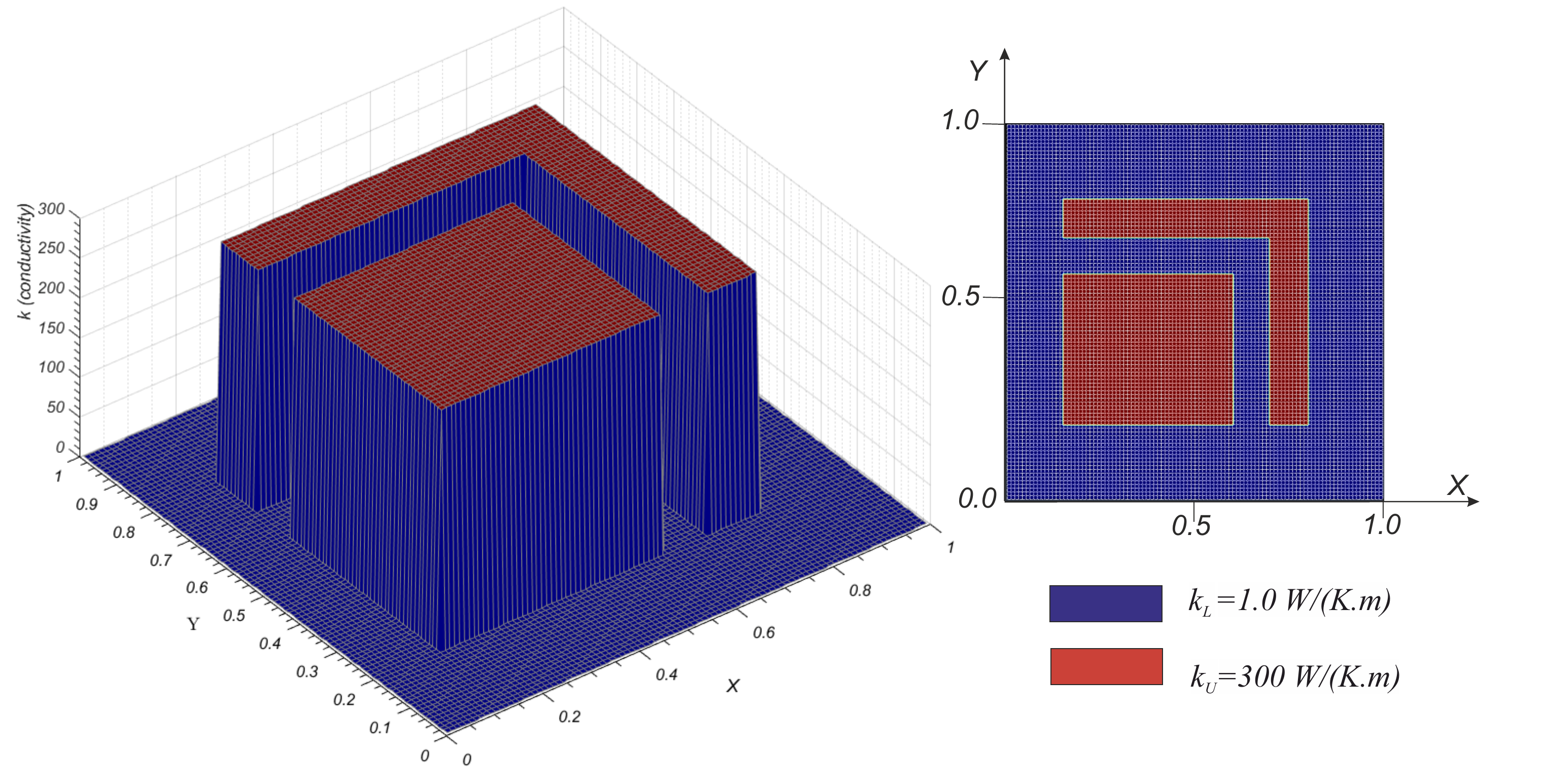} \\ \vspace{-.05in} \caption{Conductivity $k(x,y)$ used for solving the forward problem in Case I} \label{kxy-design1}
\end{figure}
\begin{figure}[h!]
    \centering
    \includegraphics[width=1\textwidth]{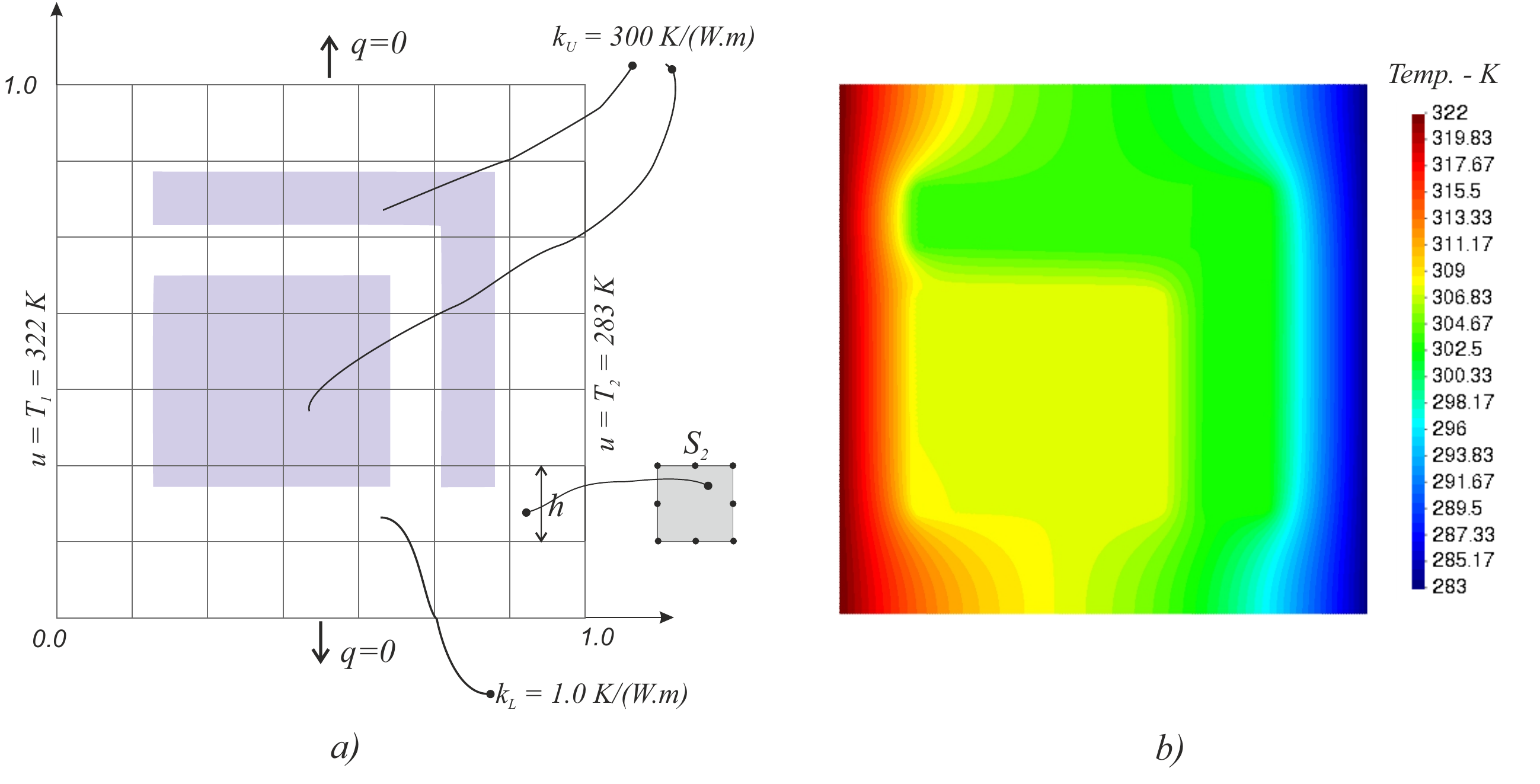} \\
    \caption{{ \textit{a)} Sketch of the distretized domain used to solve the forward problem, for case I. The finite element mesh S2 used is regular with elements size $h= 1/200$.} \textit{b)} Temperature distribution $\hat u(x,y)$ for $k(x,y)$ as in Figure \ref{kxy-design1}.}
    \label{uhat1}
\end{figure}
\bigskip
{\large \bf Setting 1:} For this case we picked $\alpha=0$ (non-penalized case) in (\ref{TP-functional}) and the set of
functions $v^r\in H^1_{\Gamma_D,0}(\Omega),\; 1\le r\le R$, was chosen as an appropriate single index reordering of the set of functions:
\begin{equation}\label{basisfunctions1}
    v^{m,n}(x,y)\;\doteq\; x^m(1-x)^n, \qquad \textrm{for } 1\le m,n\le M,
\end{equation}
where $M\in \mathbb{N}$ is appropriately chosen. Note that $v^{m,n}\in H^1_{\Gamma_D,0}(\Omega)$ as required in (\ref{eqnstar1}). Thus $R=M^2$ and $v^r_{i}\doteq
v^{m,n}(x_i,y_i)=x_{i}^m(1-x_{i})^n$ for all $1\le i\le L$, $1\le r\le R$ and $1\le m,n\le M$. The general idea behind our approach is to somehow determine ``how
much'' information about $k(x,y)$ can be obtained from the set of functions (\ref{basisfunctions1}) as a subset of $H^1_{\Gamma_D,0}(\Omega)$ through the optimality
equation (\ref{eqnstar1}) which characterizes $k(x,y)$ or, more precisely, through its discrete counterpart (\ref{discrete2}).

We first solved (\ref{TP-functional}) with $\alpha=0$. Hence, no restrictions on $k(x,y)$ were imposed. Also, we used $M=5$ in (\ref{basisfunctions1}) which results
in $R=25$ test functions in (\ref{testfunctions}). We computed the best approximate solution (i.e. the least squares solution of minimal norm) of $AK=F$ as
$K^\dagger =A^\dagger F$ where $A^\dagger$ denotes the Moore-Penrose generalized inverse of $A$ (\cite{refb:Engl-Hanke-96}). The obtained (discretized) $k(x,y)$ is
plotted in Figure \ref{obtained-kforexample1}. We observe how the basic geometric distribution of the two possible values of $k(x,y)$ is already recovered using only
the 25 functions $v^{m,n}(x,y)$ in (\ref{basisfunctions1}) with $M=5$, even without using any penalization in equation  (\ref{TP-functional}). It is timely to point
out here that the severe ill-posedness of the problem is clearly reflected on the matrix $A$ whose condition number turns out to be approximately equal to $2.5\times
10^{18}$ (here we used $L=10^4$ resulting from a $100\times 100$ regular grid on $\Omega$. It is reasonable to think that a better restoration of $k(x,y)$ could be
obtained by taking $\alpha>0$ and minimizing the functional $J_{\alpha,W}$ given in (\ref{TP-functional}) using the penalizer $W_1(K)$ defined in
(\ref{penalizerW1}).
\begin{figure}[H]
    \centering
    \includegraphics[width=1\textwidth]{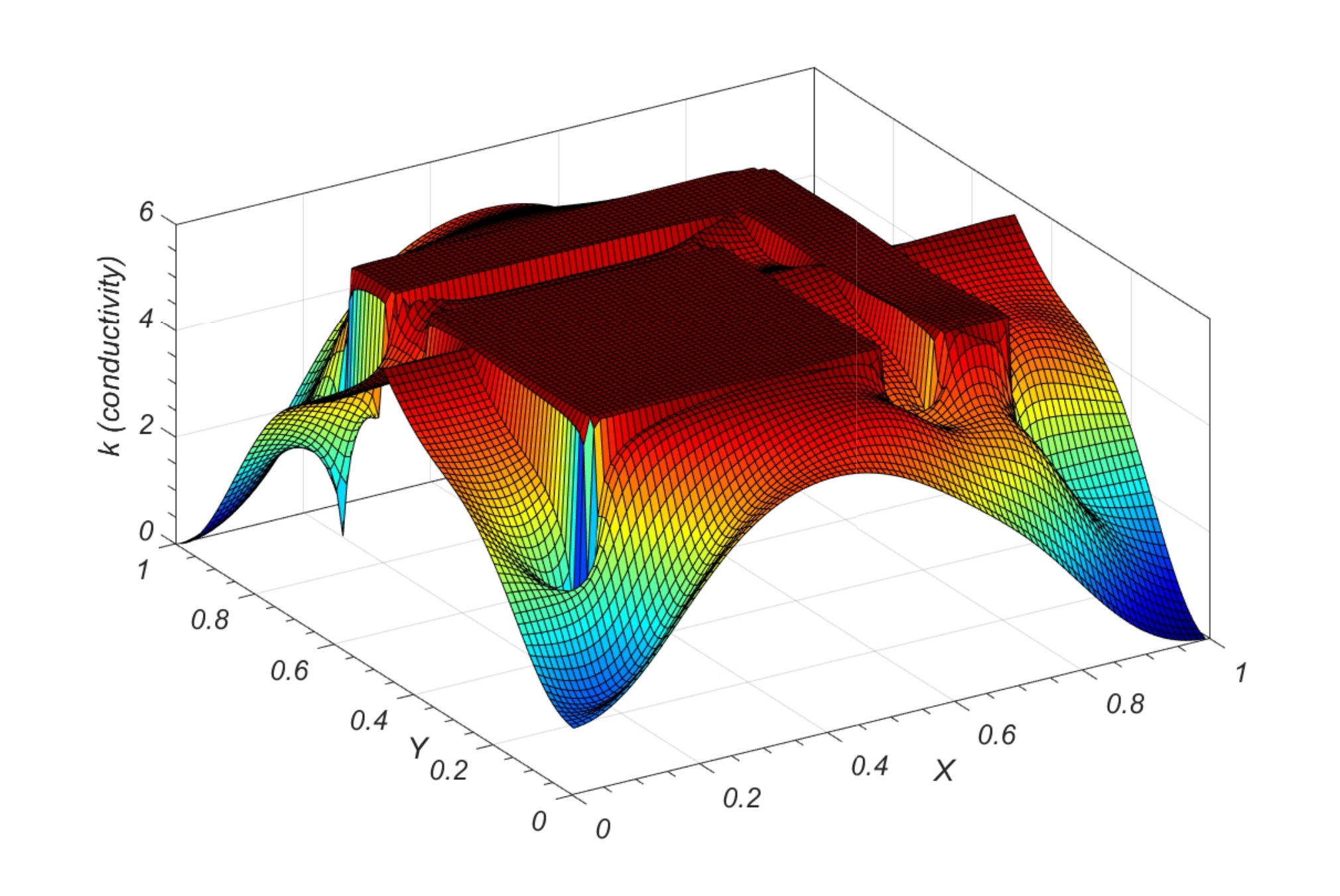} \\ \vspace{-.05in}
    \caption{Reconstruction of $k(x,y)$ obtained using $\hat u(x,y)$ as in Figure \ref{uhat1}, using a non-penalized least squares approach.}
    \label{obtained-kforexample1}
\end{figure}

\medskip
{\large \bf Setting 2:}  For this example all parameters and functions are the same as in Setting 1, except that here we used $\alpha>0$ in (\ref{TP-functional}) and
the penalizer was chosen as $W(K)=W_1(K)$ given by (\ref{penalizerW1}). The appropriate value of  $\alpha>0$ was picked by means of the L-curve method
(\cite{ref:Hansen-Oleary-93}, \cite{refb:Hansen2010}, \cite{refb:Engl-Hanke-96}). Note that finding the minimizer of (\ref{TP-functional}) corresponds to finding the
Tikhonov-Phillips $W_1$-regularized solution of (\ref{matrixform}).
The obtained conductivity $k(x,y)$ is shown in Figure \ref{obtained-kforexample2}. As it can be seen, the introduction of the penalizer $W_1$, as expected, induces
$k$ to take only values around $K_L=1.0$ and $K_U=300$. Although the correct geometric profile of $k(x,y)$, shown in Figure \ref{kxy-design1}, can be clearly
appreciated in this figure, the obtained result is far from being satisfactory. It is highly desirable to come up with a better penalizer that be able to reconstruct
the conductivity distribution as close as possible.

\begin{figure}[H]
    \centering
    \includegraphics[width=1\textwidth]{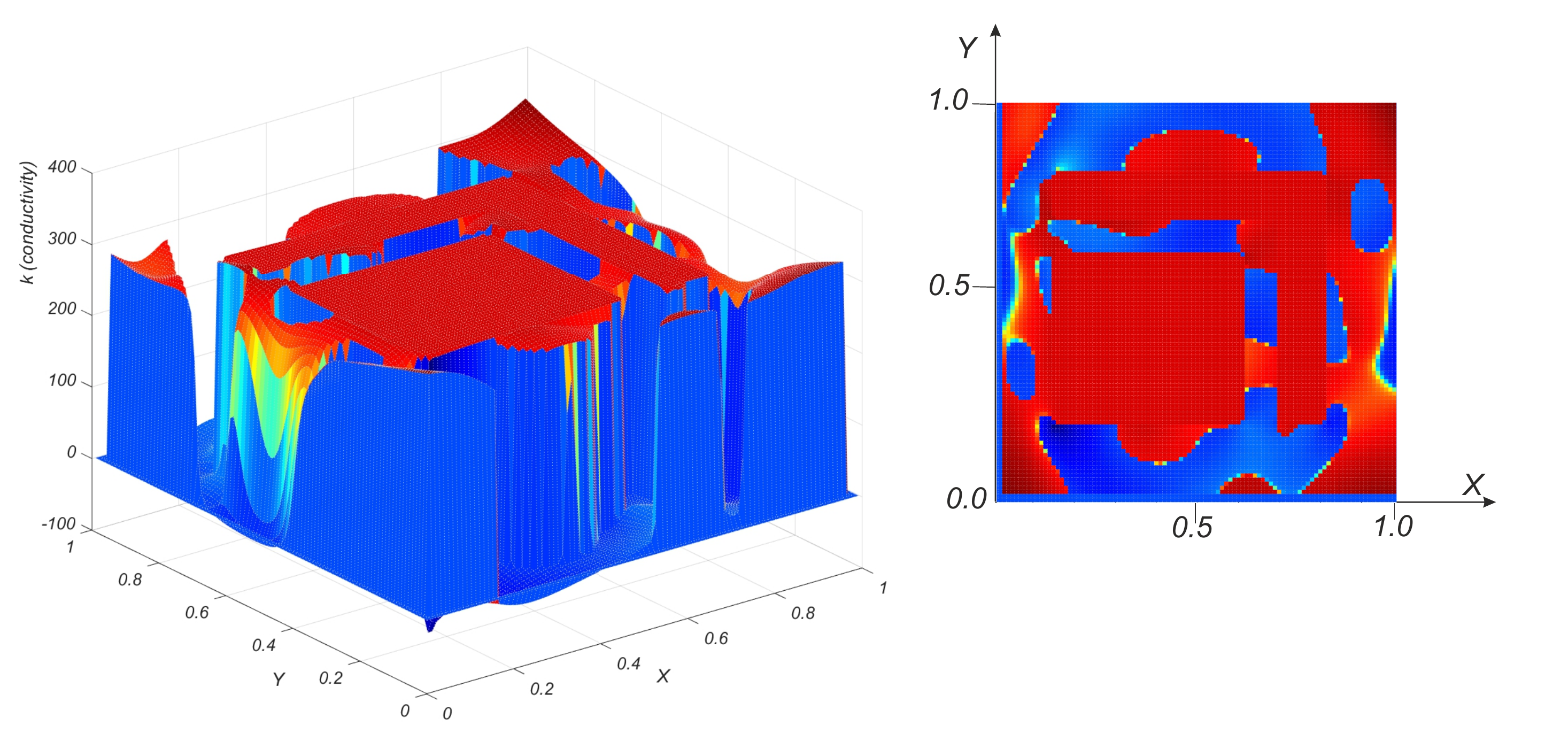} \\
    \caption{Reconstruction of $k(x,y)$ obtained using $\hat u(x,y)$ as in Figure \ref{uhat1}, via the penalized least squares approach (\ref{TP-functional}) with
        $W(K)=W_1(K)$ given as in (\ref{penalizerW1}).} \label{obtained-kforexample2}
\end{figure}

\medskip
{\large \bf Setting 3:} Here we proceed to use the penalizer $W_2(K)$ which, unlike $W_1$, includes data-driven information about the local size of $\|\nabla\hat
u(x,y)\|$. Hence, $k(x,y)$ is now obtained as the Tikhonov-Phillips $W_2$-regularized solution of problem (\ref{matrixform}), that is, as the global minimizer of
(\ref{TP-functional}) when the penalizer $W(K)$ is given by $W_2(K)$ as defined in (\ref{penalizerW2}). The value of the threshold parameter $\gamma$ needed to
define $b_U$ in (\ref{penalizerW2}) was set to $0.01*M$, where $M\doteq \max_{1\le i\le L} \|\hat u(x_i,y_i)\|$ while the regularization parameter $\alpha$ was again
chosen by means of the L-curve method. The restored conductivity profile $k(x,y)$ is shown in Figure \ref{obtained-kforexample3}. As it can be observed, although
some small artifacts still appear, a much better approximation of the true profile is obtained.
\begin{figure}[H]
    \centering
    \includegraphics[width=1\textwidth]{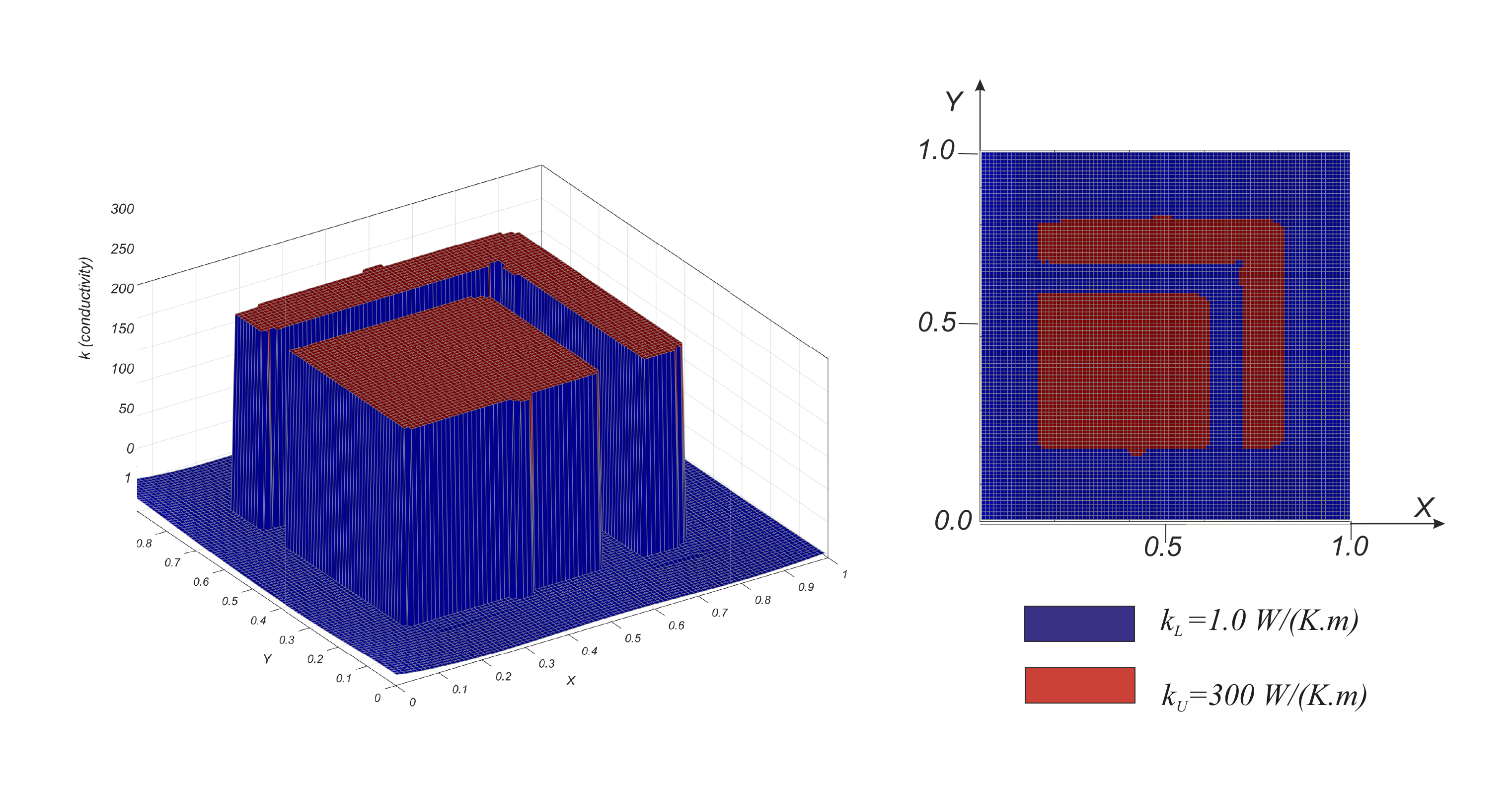} \\
    \caption{Reconstruction of $k(x,y)$ obtained using $\hat u(x,y)$ as in Figure \ref{uhat1}, using a penalized least squares approach with $W(K)=W_2(K)$ defined as in
        (\ref{penalizerW2})} \label{obtained-kforexample3}
\end{figure}

We now proceed to present the results obtained for three other different configurations of the conductivity profile $k(x,y)$. For the sake of brevity, we shall only
show the results obtained under Setting 3 of Case I, which is clearly by far the best approach for solving the inverse problem at hand.

\bigskip
\subsubsection{Case II} For this case the conductivity profile is as depicted in Figure \ref{kxy-design2}. Here again we computed $\hat
u(x,y)$ by first running the forward problem with $T_1 = 318.15\; [K]$, $T_2= 288.15\; [K]$, $c(x,y)=1.0=$\,constant. The resulting temperature distribution $\hat
u(x,y)$, whose discretized values are used as inputs for the inverse problem, is shown in Figure \ref{uhat2}\,b).

\begin{figure}[h!]
    \centering
    \includegraphics[width=1
    \textwidth]{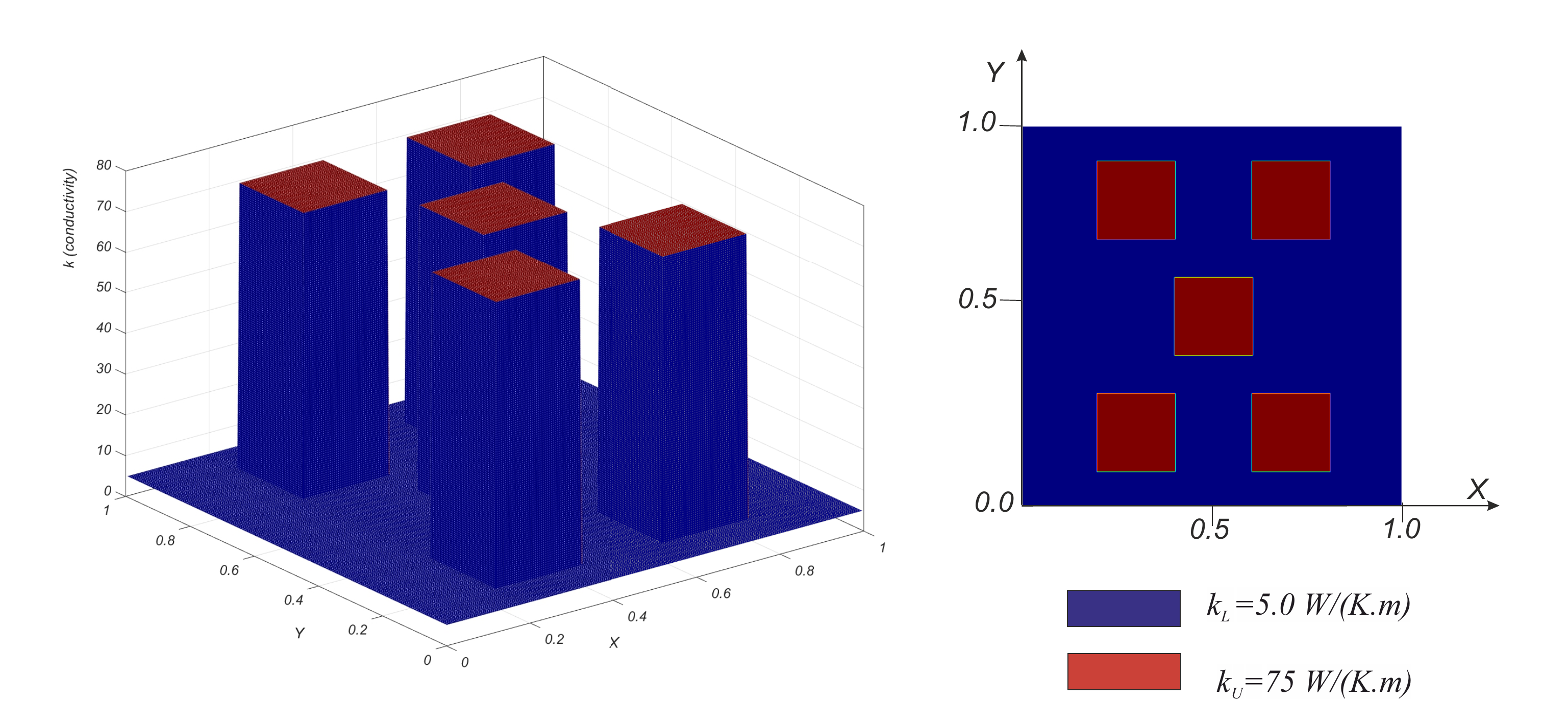} \\  \caption{Distributed values of the conductivity $k(x,y)$ used for solving the forward problem in Case II} \label{kxy-design2}
\end{figure}
\begin{figure}[h!]
    \centering
    \includegraphics[width=1\textwidth]{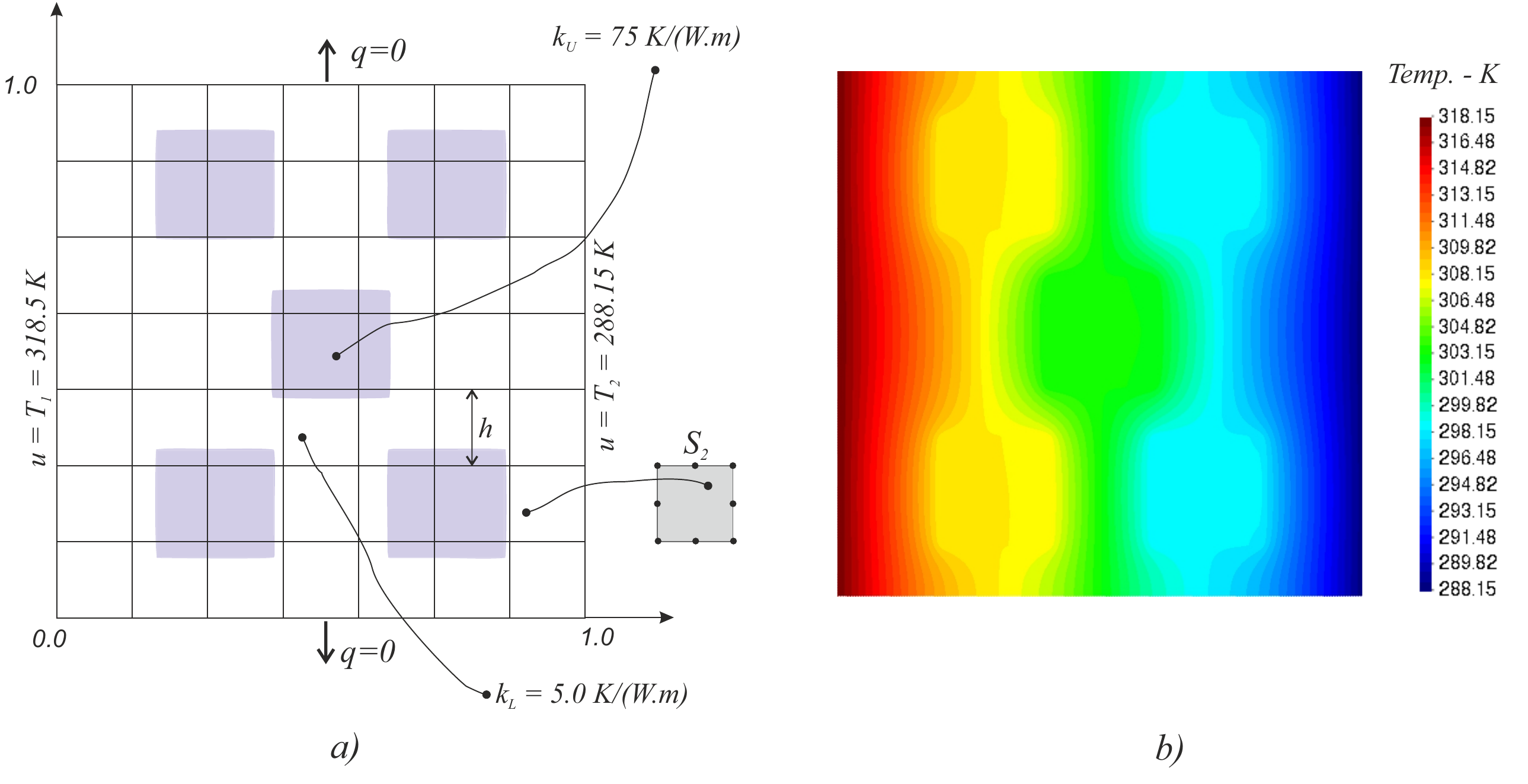} \\
    \caption{{ \textit{a)} Sketch of the distretized domain used to solve the forward problem for Case II. The finite element mesh S2 used is regular with elements size $h= 1/200$.} \textit{b)} Temperature distribution $\hat u(x,y)$ for $k(x,y)$ as in Figure \ref{kxy-design2}, Case II.} \label{uhat2}
\end{figure}
\bigskip Here, the set of functions $v^r\in H^1_{\Gamma_D,0}(\Omega),\; 1\le r\le R$, was
chosen as in Setting 1 of Case I. We proceeded to compute the generalized Tikhonov-Phillips solution of problem (\ref{matrixform}) when the penalizer is given by
$W_2(K)$ as defined in (\ref{penalizerW2}). The threshold parameter was now chosen as $\gamma=0.0125\,M$, where, as before, $M\doteq \max_{1\le i\le L} \|\hat
u(x_i,y_i)\|$ while the regularization parameter $\alpha$ was again chosen by means of the L-curve method. The restored conductivity profile $k(x,y)$ is shown in
Figure \ref{obtained-k-forcaseII-3}. Here again, despite a few artifacts, we can observe how our method is able to satisfactory reconstruct the conductivity
distribution profile.

\begin{figure}[h!]
    \centering
    \includegraphics[width=1\textwidth]{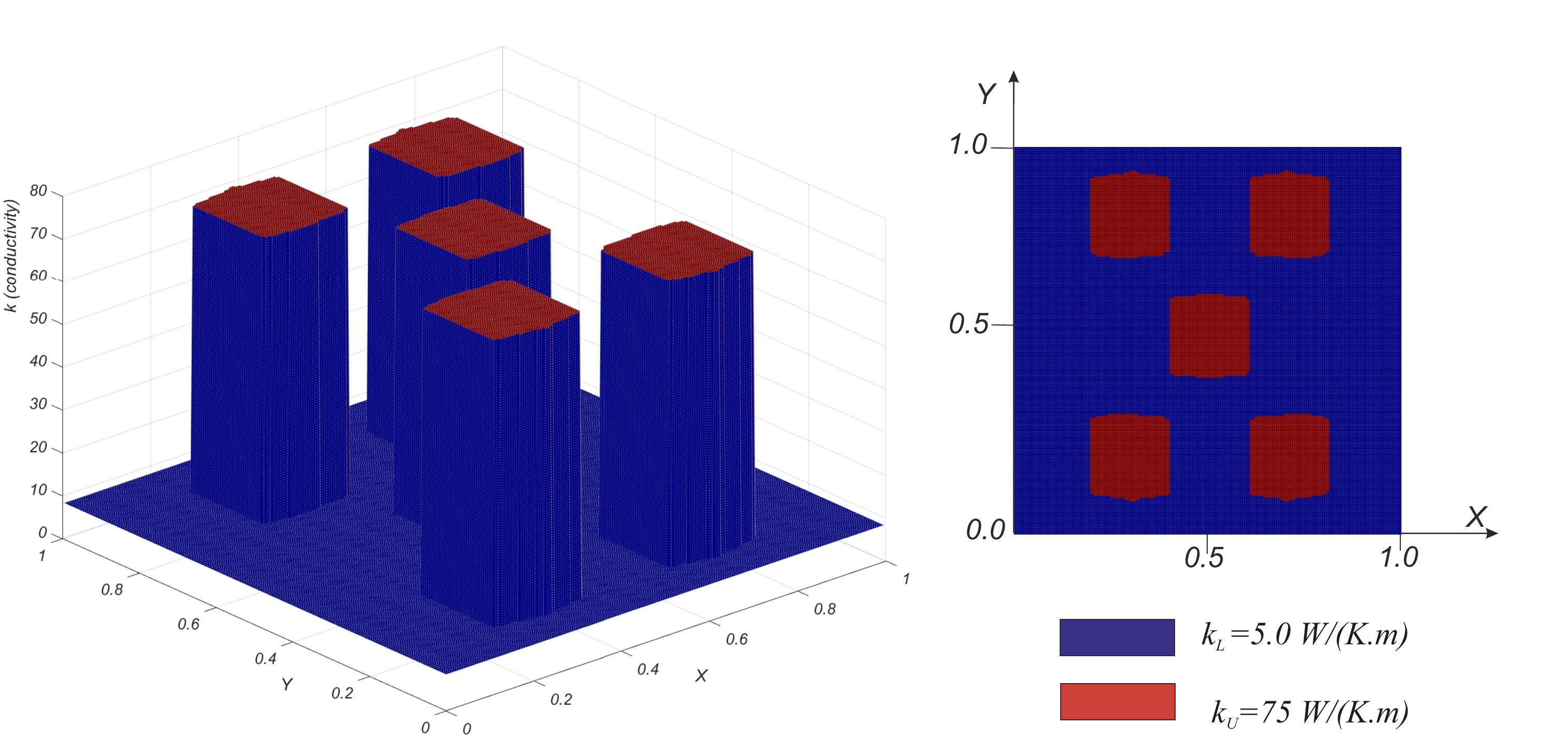} \\
    \caption{Reconstruction of $k(x,y)$ obtained using $\hat u(x,y)$ as in Figure \ref{uhat2}, using a penalized least squares approach with $W_2(K)$ defined as in
        (\ref{penalizerW2}), Case II.} \label{obtained-k-forcaseII-3}
\end{figure}

\bigskip
\subsubsection{Case III} For this case the conductivity profile is as
depicted in Figure \ref{kxy-design3}. Here again we computed $\hat u(x,y)$ by first running the forward problem with $T_1 = 373.15\; [K]$, $T_2= 353.15\; [K]$,
$c(x,y)=1.0=$\,constant, $k_U=100$ and $k_L=0.7$. The resulting temperature distribution $\hat u(x,y)$ is shown in Figure \ref{uhat3}\,b).

\begin{figure}[h!]
    \centering
    \includegraphics[width=1
    \textwidth]{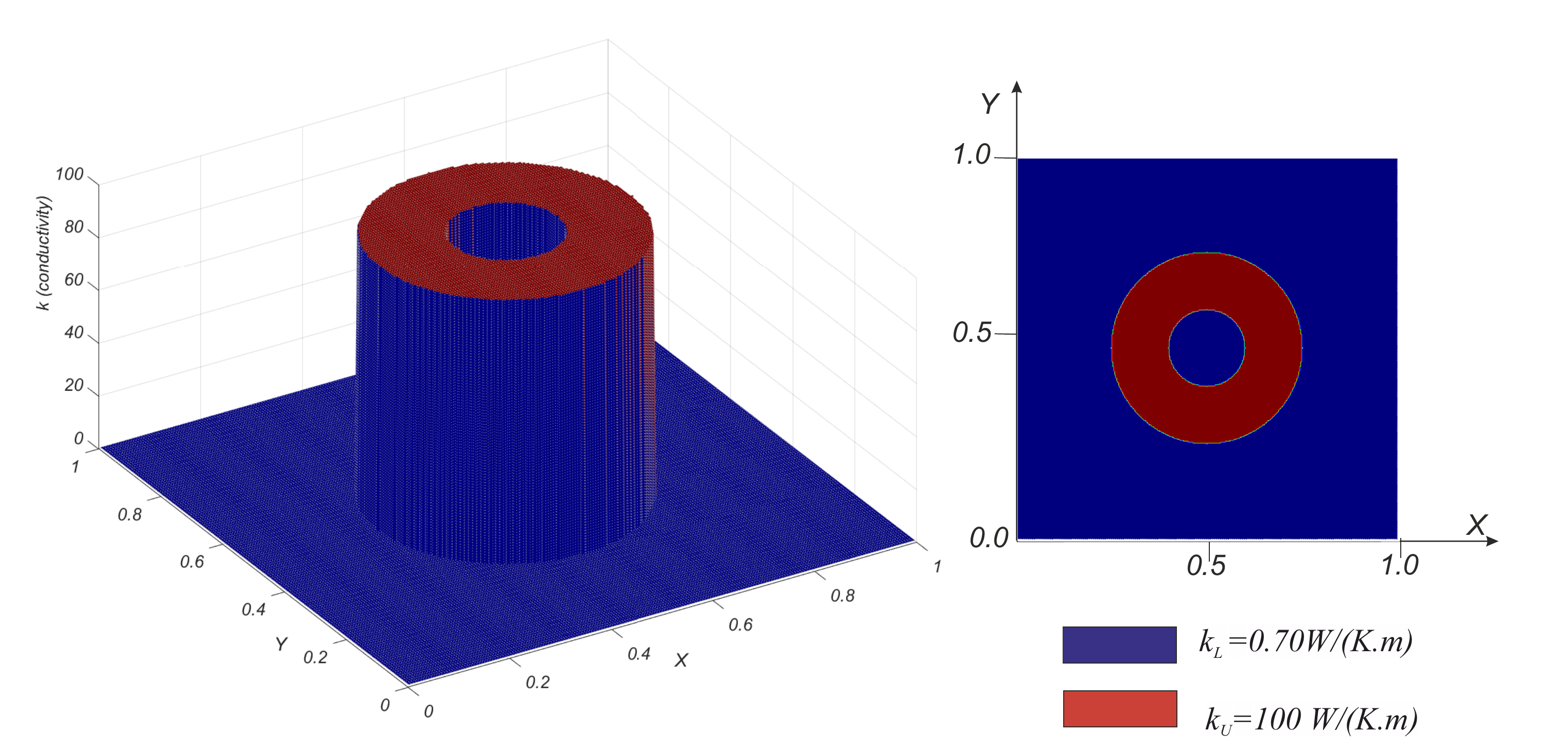} \\  \caption{Distributed values of the conductivity $k(x,y)$ used for solving the forward problem in Case III} \label{kxy-design3}
\end{figure}
\begin{figure}[h!]
    \centering
    \includegraphics[width=1\textwidth]{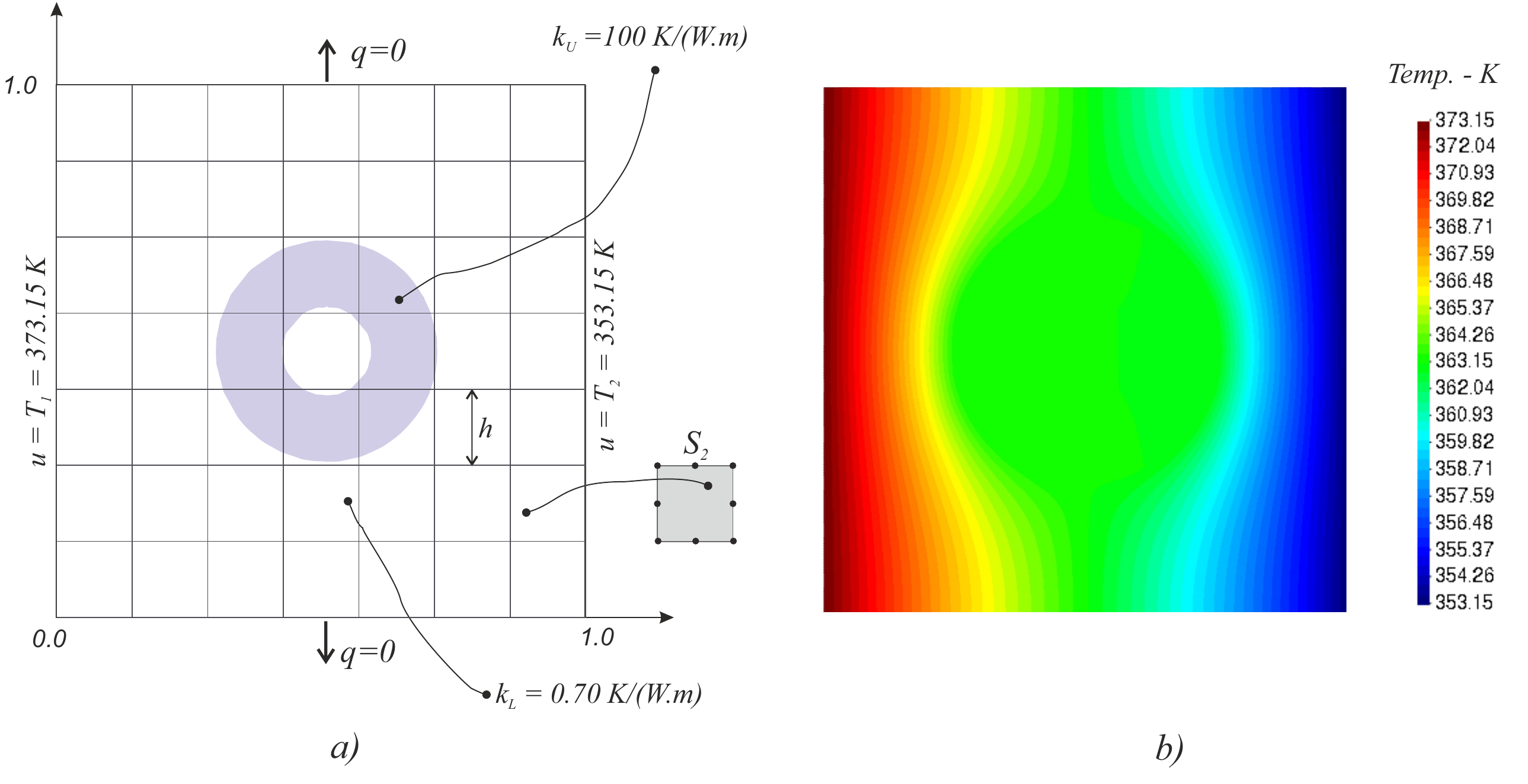} \\
    \caption{{ \textit{a)} Sketch of the discretized domain used to solve the forward problem, for case III. The finite element mesh S2 used is regular with elements size $h= 1/200$.}
    \textit{b)}  Temperature distribution $\hat u(x,y)$ for $k(x,y)$ as in Figure \ref{kxy-design3}, Case III.} \label{uhat3}
\end{figure}

\bigskip
As in the previous case, the set of functions $v^r\in H^1_{\Gamma_D,0}(\Omega),\; 1\le r\le R$, was
chosen as in Setting 1 of Case I. We proceeded to compute the generalized Tikhonov-Phillips solution of problem (\ref{matrixform}) when the penalizer is given by
$W_2(K)$ as defined in (\ref{penalizerW2}). The value of the threshold was now chosen as $\gamma=0.0125\,M$, where, as before, $M\doteq \max_{1\le i\le L} \|\hat
u(x_i,y_i)\|$ while the regularization parameter was again chosen by means of the L-curve method. The restored conductivity profile $k(x,y)$ is shown in Figure
\ref{obtained-k-forcaseIII-3}. Once again we can observe how our method is able to satisfactory reconstruct the conductivity distribution profile.

\begin{figure}[h!]
    \centering
    \includegraphics[width=0.95\textwidth]{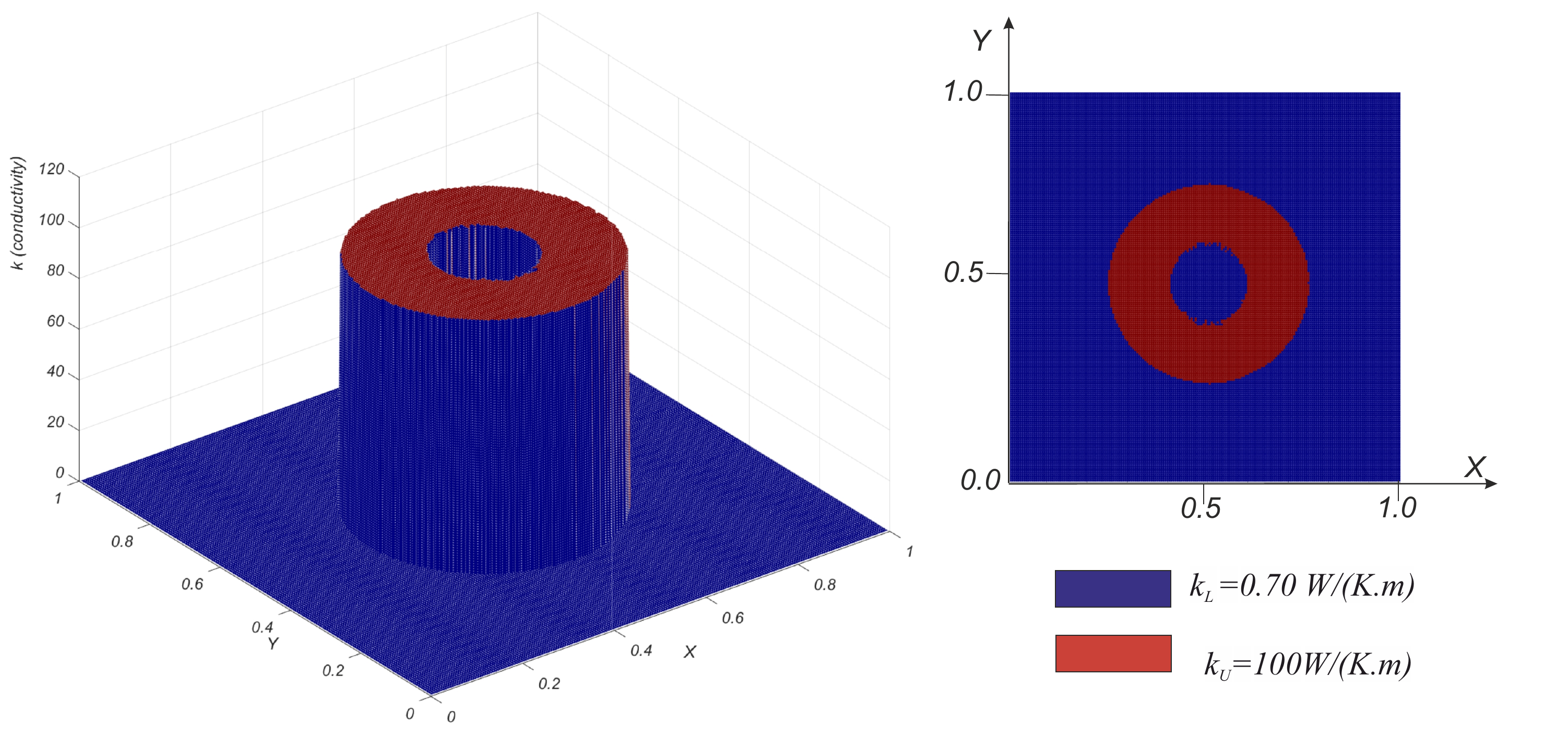} \\ \vspace{-.05in}
    \caption{Reconstruction of $k(x,y)$ obtained using $\hat u(x,y)$ as in Figure \ref{uhat3}\,b), using a penalized least squares approach with $W_2(K)$ defined as in
        (\ref{penalizerW2}), Case III.} \label{obtained-k-forcaseIII-3}
\end{figure}

\bigskip
\subsubsection{Case IV} We end up presenting a case where the conductivity profile has a very irregular shape, as depicted in Figure \ref{kxy-design4}. The goal is
to verify if the methods is able to detect these irregular boundaries in $k(x,y)$. Here again we computed $\hat u(x,y)$ by first running the forward problem with
$T_1 = 308.15\; [K]$, $T_2= 298.15\; [K]$, $c(x,y)=1.0=$\,constant, $k_U=125$ and $k_L=20$. The resulting temperature distribution $\hat u(x,y)$ is shown in Figure
\ref{uhat_4}\,b).

\begin{figure}[h!]
    \centering
    \includegraphics[width=1
    \textwidth]{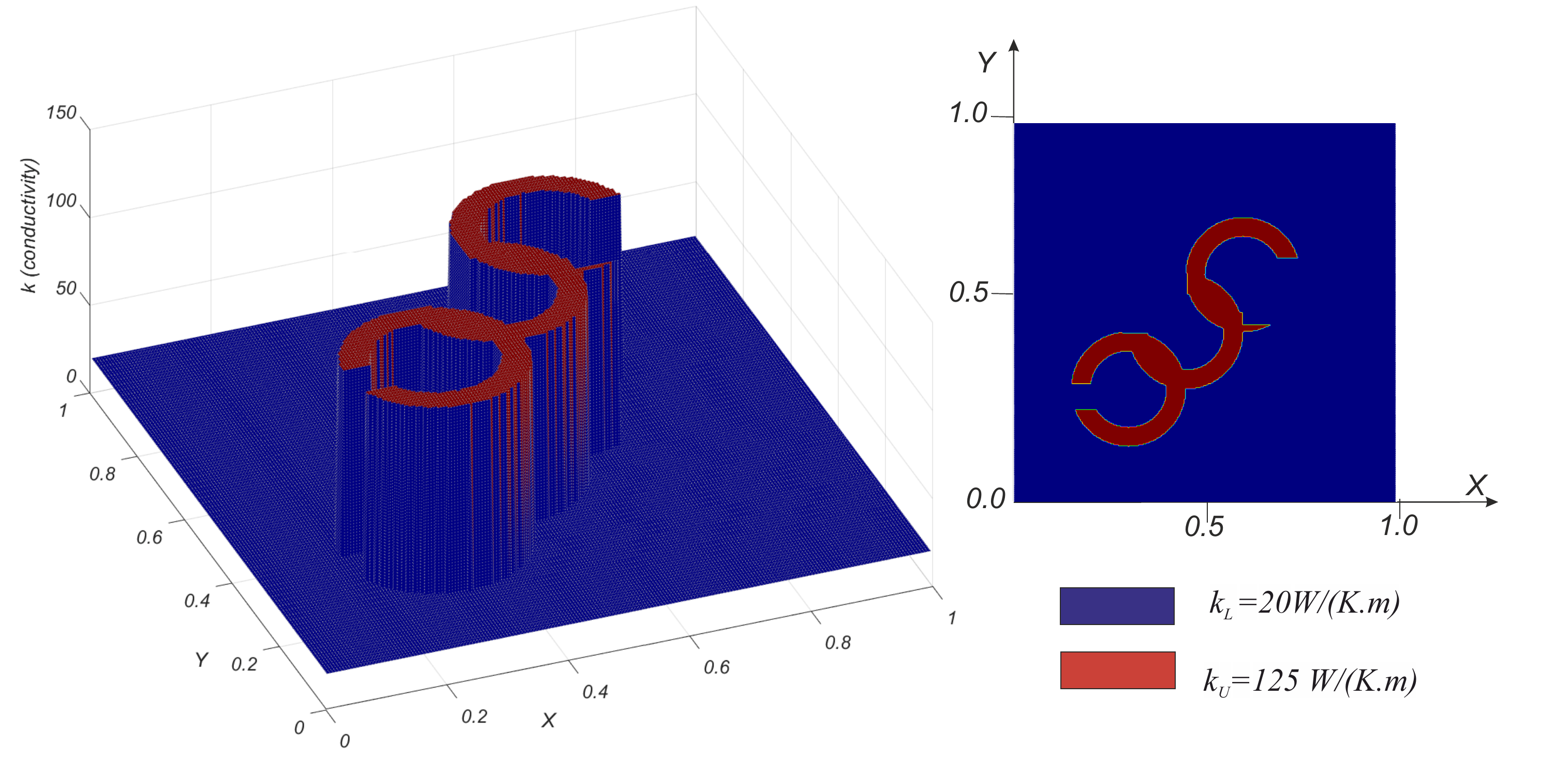} \\ \vspace{-.1in} \caption{Distributed values of the conductivity $k(x,y)$ used for solving the forward problem in Case IV.} \label{kxy-design4}
\end{figure}
%
%
%
As in the previous cases, we proceeded to compute the generalized Tikhonov-Phillips solution of problem (\ref{matrixform}), as the global minimizer of (\ref{TP-functional})
when the penalizer is given by $W_2(K)$ as defined in (\ref{penalizerW2}). The value of the threshold in this case was chosen as $\gamma=0.3153\,M$, with $M\doteq
\max_{1\le i\le L} \|\hat u(x_i,y_i)\|$ while the regularization parameter was again chosen by means of the L-curve method. The restored conductivity profile
$k(x,y)$ is shown in Figure \ref{obtained-k-forcaseIV-3}. Once again, we can observe how our method is able to quite satisfactory reconstruct the conductivity
distribution profile.

\begin{figure}[h!]
    \centering
    \includegraphics[width=1\textwidth]{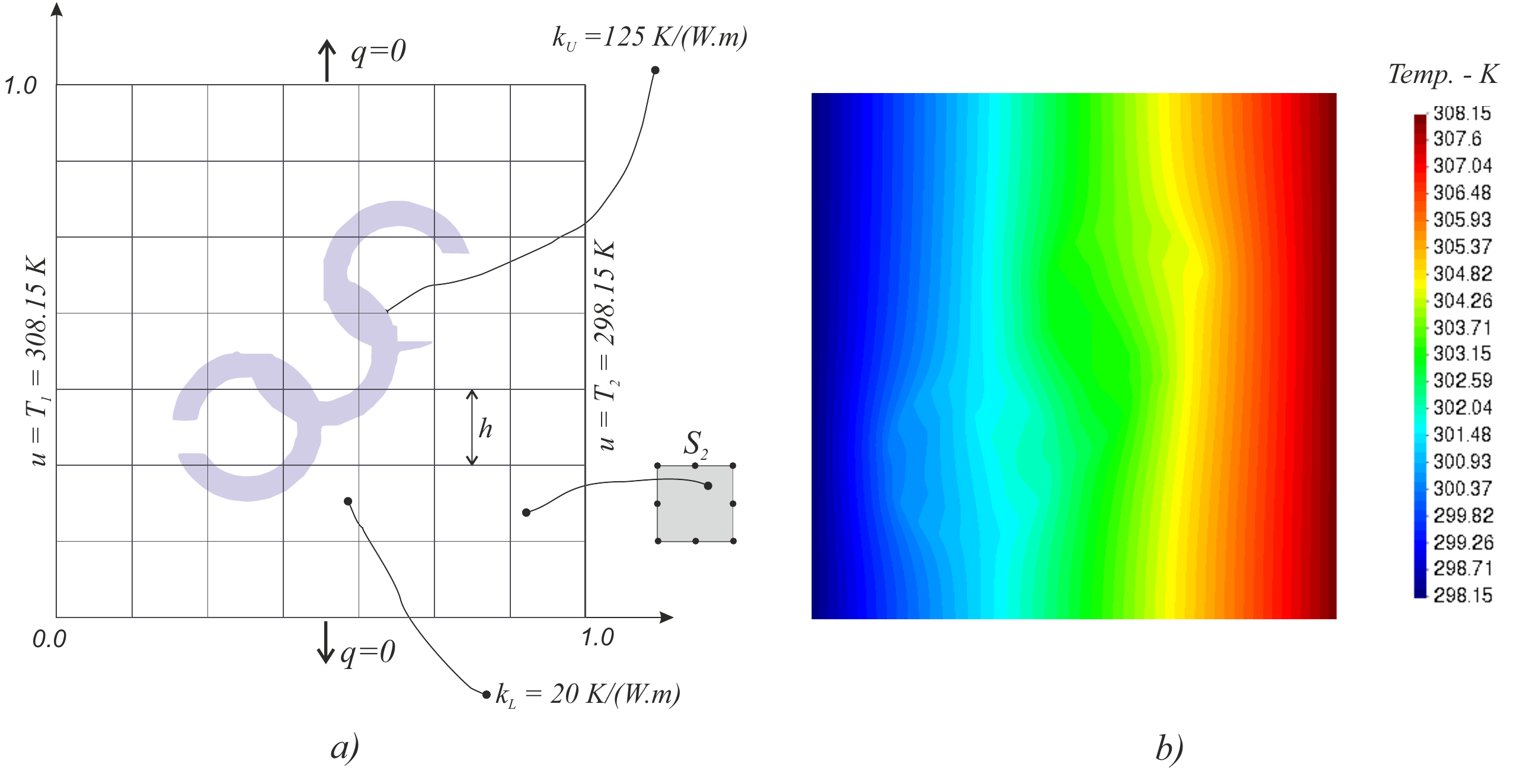} \\
    \caption{ { \textit{a)} Sketch of the discretized  domain used to solve the forward problem, for case IV. The finite element mesh S2 used is regular with elements size $h= 1/200$.}
    \textit{b)} Temperature distribution $\hat u(x,y)$ for $k(x,y)$ as in Figure \ref{kxy-design4}, Case IV.} \label{uhat_4}
\end{figure}

\begin{figure}[h!]
    \centering
    \includegraphics[width=1\textwidth]{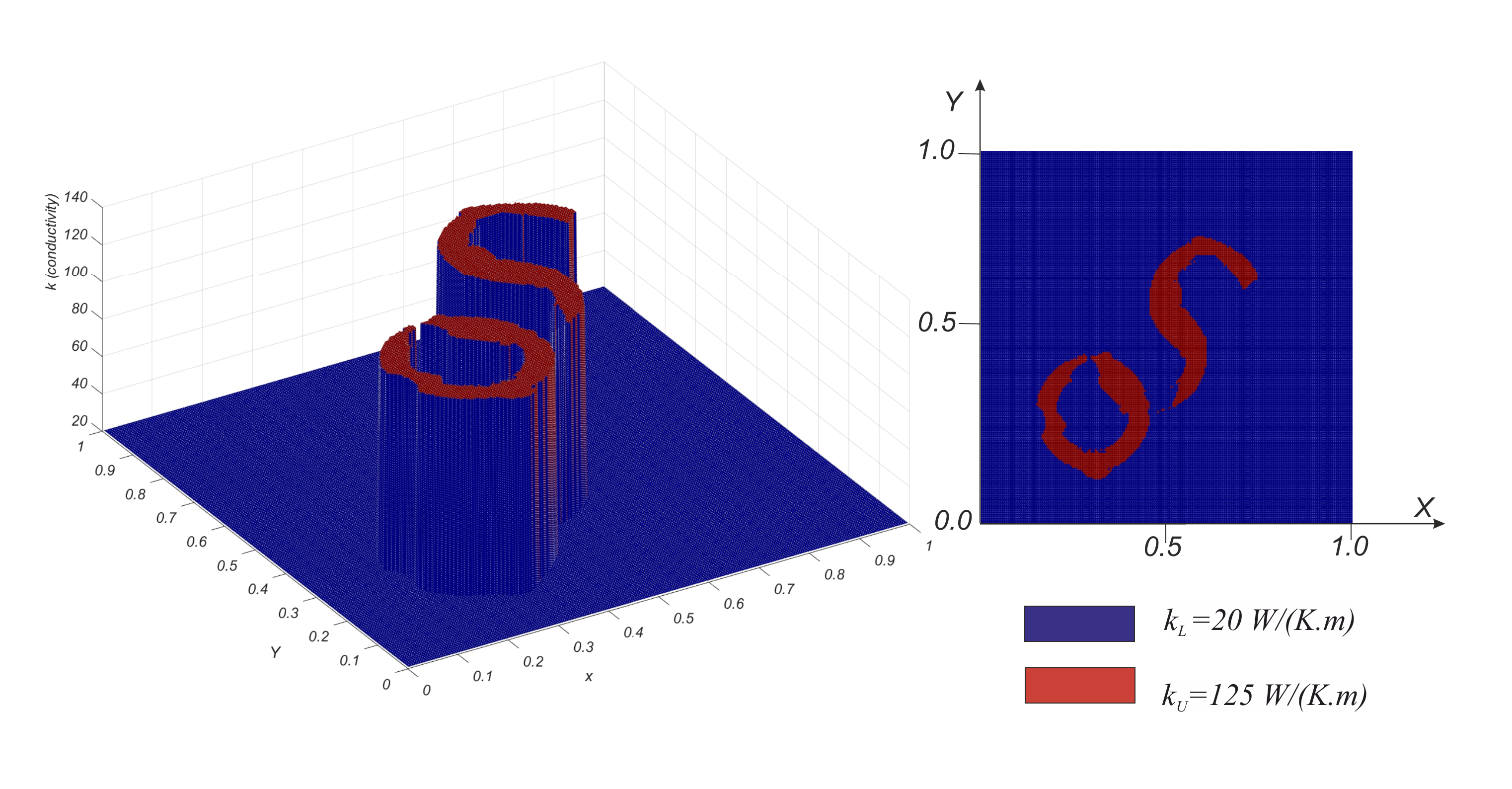} \\
    \caption{Reconstruction of $k(x,y)$ obtained using $\hat u(x,y)$ as in Figure \ref{uhat_4}, using a penalized least squares approach with $W_2(K)$ defined as in
        (\ref{penalizerW2}), Case IV.} \label{obtained-k-forcaseIV-3}
\end{figure}

\section*{Conclusions and future work}

In this article we introduced a method for solving the inverse problem of estimating the principal coefficient in a steady state elliptic diffusion equation. The
method is based on the discretization of an optimality equation derived from a variational approach which leads to a least squares problem based on a  fidelity term.
The method is followed by regularization consisting of adding an appropriate penalizer, encoding prior-information about the conductivity profile, leading to a the
problem of finding the minimizer of a generalized Tikhonov-Phillips functional.

Several numerical experiments with different types of discontinuous distributions for the leading coefficient function $k(x,y)$ were presented, showing that, with
the appropriate penalizer, the method yields excellent reconstructions of the true (unknown in practical problems) function.

There is clearly much room for improvements and generalizations. First we find it timely to point out that several numerical experiments have shown that adding more
functions in $H^1_{\Gamma_D,0}(\Omega)$ to the the set in (\ref{basisfunctions1}) does not improve the reconstructions in none of the cases. That is due to the
severe ill-posedness of the problem.

As it was shown, the success of the method is highly dependent on the appropriate choice of the regularization term. The penalizer $W_2$ defined in
(\ref{penalizerW2}) works very well for the case in which $k(x,y)$ takes only two prescribed values. However, in that case, although experiments showed that the
results are somewhat robust with respect to the choice of $\gamma$, designing rigorous analytic ways for estimating an optimal value of that threshold parameter
$\gamma$ is highly desirable.

A detailed analysis of the results and images show that the small artifacts obtained in the reconstruction of the function $k(x,y)$ are always associated to places
where $|\nabla u|=0$. Although this is reasonable since at those points $k$ cannot be restored, more research is necessary to come up with a way to avoid those
artifacts, perhaps taking into account a-priori information about the expected conductivity.

In regard to the design of appropriate penalizers, no general recipe is expected to be found which will work for arbitrary functions $k(x,y)$. It is not completely
clear, for instance, what would be a good choice for the penalizer $W(\cdot)$ in the case in which three (or more) possible conductivity values are present. Efforts
in all of these directions are currently under way.

\section*{Acknowledgements}
 This work was supported in part by Consejo Nacional de Investigaciones Cient\'{\i}ficas y T\'{e}cnicas, CONICET, through grants PIP 2021-2023 number 11220200100806CO and PUE-IMAL number 22920180100041CO, by Agencia Nacional de Promoci\'{o}n de la Investigaci\'{o}n,
el Desarrollo Tecnol\'{o}gico y la Innovaci\'{o}n, througn grant PICT-2019-2019-01933 and by Universidad Nacional del Litoral, through grant CAI+D 2020 Nro.
50620190100069LI.




\bibliographystyle{Abbrv}

\bibliography{ref1,ref_2}

\begin{thebibliography}{10}

\bibitem{Alessandrini-1984}
G.~Alessandrini.
\newblock On the identification of the leading coefficient of an elliptic
  equation.
\newblock {\em Pubblicazioni dell'Istituto di analisi globale e applicazioni.},
  1984.

\bibitem{Alessandrini-1986}
G.~Alessandrini.
\newblock An identification problem for an elliptic equation in two variables.
\newblock {\em Annali di Matematica pura ed applicata}, (145):265–295, 1986.

\bibitem{refb:Bear-1972}
J.~Bear.
\newblock {\em Dynamics of Fluids in Porous Media}.
\newblock American Elsevier, New York, 1972.

\bibitem{Bendsoe2003}
M.~P. Bends{\o}e and O.~Sigmund.
\newblock {\em Topology optimization. Theory, methods, and applications}.
\newblock Springer-Verlag, 2003.

\bibitem{ref:Bongiorno-Valente-1977}
F.~Bongiorno and V.~Valente.
\newblock A method of characteristics for solving an underground water maps
  problem.
\newblock {\em Pubblicazioni Istituto per le Applicazioni del Calcolo "Mauro
  Picone". III}, 116, 1977.

\bibitem{Bongiorno-Valente-1977}
F.~Bongiorno and V.~Valente.
\newblock {\em A Method of Characteristics for Solving an Underground Water
  Maps Problem}.
\newblock Pubblicazioni (Istituto per le applicazioni del calcolo "Mauro
  Picone"). IAC, 1977.

\bibitem{Calderon-1980}
A.~P. Calderón.
\newblock {On an inverse boundary value problem}.
\newblock {\em {Computational \& Applied Mathematics}}, 25:133 -- 138, 00 2006.

\bibitem{refb:Engl-Hanke-96}
H.~W. Engl, M.~Hanke, and A.~Neubauer.
\newblock {\em Regularization of inverse problems}, volume 375 of {\em
  Mathematics and its Applications}.
\newblock Kluwer Academic Publishers Group, Dordrecht, 1996.

\bibitem{fachinotti2018}
V.~D. Fachinotti, {\'A}.~A. Ciarbonetti, I.~Peralta, and I.~Rintoul.
\newblock Optimization-based design of easy-to-make devices for heat flux
  manipulation.
\newblock {\em International Journal of Thermal Sciences}, 128:38--48, 2018.

\bibitem{fachinotti016}
V.~D. Fachinotti, I.~Peralta, A.~E. Huespe, and A.~A. Ciarbonetti.
\newblock Control of heat flux using computationally designed metamaterials.
\newblock In {\em EngOpt 2016 -- 5th International Conference on Engineering
  Optimization}, Iguassu Falls, Brasil, 2016.

\bibitem{ref:Hadamard-1902}
J.~Hadamard.
\newblock Sur les probl{\`e}mes aux d{\'e}riv{\'e}es partielles et leur
  signification physique.
\newblock {\em Princeton University Bulletin}, 13:49--52, 1902.

\bibitem{refb:Hansen2010}
P.~C. Hansen.
\newblock {\em Discrete Inverse Problems: Insight and Algorithms}, volume FA07
  of {\em Fundamentals of Algorithms}.
\newblock Society for Industrial and Applied Mathematics, Philadelphia, 2010.

\bibitem{ref:Hansen-Oleary-93}
P.~C. Hansen and D.~P. O'Leary.
\newblock The use of the {$L$}-curve in the regularization of discrete
  ill-posed problems.
\newblock {\em SIAM J. Sci. Comput.}, 14(6):1487--1503, 1993.

\bibitem{huang2000two}
C.-H. Huang and S.-C. Chin.
\newblock A two-dimensional inverse problem in imaging the thermal conductivity
  of a non-homogeneous medium.
\newblock {\em International Journal of Heat and Mass Transfer},
  43(22):4061--4071, 2000.

\bibitem{convex2004}
M.~Janicki and A.~Napieralski.
\newblock Inverse heat conduction problems in electronics with special
  consideration of analytical analysis methods.
\newblock In {\em 2004 International Semiconductor Conference. CAS 2004
  Proceedings (IEEE Cat. No.04TH8748)}, volume~2, pages 455--458 vol.2, 2004.

\bibitem{JANICKI199851}
M.~Janicki, M.~Zubert, and A.~Napieralski.
\newblock Application of inverse heat conduction methods in temperature
  monitoring of integrated circuits.
\newblock {\em Sensors and Actuators A: Physical}, 71(1):51--57, 1998.

\bibitem{Knowles-Wallace-1996}
I.~Knowles and R.~Wallace.
\newblock A variational solution for the aquifer transmissivity problem.
\newblock {\em Inverse Problems}, 12(6):953--963, dec 1996.

\bibitem{Kohn-Vogelius-1984}
R.~Kohn and M.~Vogelius.
\newblock Determining conductivity by boundary measurements.
\newblock {\em Communications on Pure and Applied Mathematics}, 37(3):289--298,
  1984.

\bibitem{Kohn-Lowe1988}
R.~V. Kohn and B.~D. Lowe.
\newblock A variational method for parameter identification.
\newblock {\em ESAIM: Mathematical Modelling and Numerical Analysis},
  22(1):119--158, 1988.

\bibitem{Kohn-Vogelius-1985}
R.~V. Kohn and M.~Vogelius.
\newblock Determining conductivity by boundary measurements ii. interior
  results.
\newblock {\em Communications on Pure and Applied Mathematics}, 38(5):643--667,
  1985.

\bibitem{Littman-Stampacchia-Weinberger-1963}
W.~Littman, G.~Stampacchia, and H.~F. Weinberger.
\newblock Regular points for elliptic equations with discontinuous
  coefficients.
\newblock {\em Annali della Scuola Normale Superiore di Pisa-Classe di
  Scienze}, 17(1-2):43--77, 1963.

\bibitem{MIERZWICZAK2011}
M.~Mierzwiczak and J.~Kołodziej.
\newblock The determination temperature-dependent thermal conductivity as
  inverse steady heat conduction problem.
\newblock {\em International Journal of Heat and Mass Transfer},
  54(4):790--796, 2011.

\bibitem{PERALTA2017}
I.~Peralta, V.~D. Fachinotti, and {\'A}.~A. Ciarbonetti.
\newblock Optimization-based design of a heat flux concentrator.
\newblock {\em Scientific Reports}, 7(40591):1--8, 2017.

\bibitem{Richter-1981}
G.~R. Richter.
\newblock An inverse problem for the steady state diffusion equation.
\newblock {\em SIAM Journal on Applied Mathematics}, 41(2):210--221, 1981.

\bibitem{bendose2003_1}
H.~Rodrigues, J.~M. Guedes, and M.~P. Bends{\o}e.
\newblock Hierarchical optimization of material and structure.
\newblock {\em Structural and Multidisciplinary Optimization}, 24(1):1--10,
  2002.

\bibitem{rodriguez2012inverse}
F.~L. Rodr{\'\i}guez and V.~de~Paulo~Nicolau.
\newblock Inverse heat transfer approach for ir image reconstruction:
  Application to thermal non-destructive evaluation.
\newblock {\em Applied thermal engineering}, 33:109--118, 2012.

\bibitem{Stampacchia-1966}
G.~Stampacchia.
\newblock Èquations elliptiques du second ordre à coefficients discontinus.
\newblock {\em Sèminaire Jean Leray}, (3):1--77, 1963-1964.

\bibitem{Silvester-Uhlmann-1986}
J.~Sylvester and G.~Uhlmann.
\newblock A uniqueness theorem for an inverse boundary value problem in
  electrical prospection.
\newblock {\em Communications on Pure and Applied Mathematics}, 39(1):91--112,
  1986.

\bibitem{Silvester-Uhlmann-1987}
J.~Sylvester and G.~Uhlmann.
\newblock A global uniqueness theorem for an inverse boundary value problem.
\newblock {\em Annals of Mathematics}, 125(1):153--169, 1987.

\bibitem{Yakowitz-Duckstein-1980}
S.~Yakowitz and L.~Duckstein.
\newblock Instability in aquifer identification: Theory and case studies.
\newblock {\em Water Resources Research}, 16(6):1045--1064, 1980.

\bibitem{Yeh-1986}
W.~W.-G. Yeh.
\newblock Review of parameter identification procedures in groundwater
  hydrology: The inverse problem.
\newblock {\em Water resources research}, 22(2):95--108, 1986.

\bibitem{chinyu98}
C.~yu~Yang.
\newblock Estimation of the temperature-dependent thermal conductivity in
  inverse heat conduction problems.
\newblock {\em Applied Mathematical Modelling}, 23:469,478, 1999.

\end{thebibliography}


\end{document}